\documentclass[12pt,reqno]{amsart}
\usepackage{amsfonts}
\usepackage{indentfirst}
\usepackage{graphicx}
\usepackage{amssymb}
\usepackage{color,xcolor}
\usepackage{graphicx,epstopdf}
\usepackage{epsfig}
\usepackage{subfig}
\usepackage{caption}
\usepackage{mathrsfs}
\usepackage{bbm}
\usepackage{chngpage}
\usepackage{cite}
\usepackage{multirow}
\usepackage{multicol}
\usepackage{dsfont}
\usepackage{bm}
\usepackage{amssymb,amsmath,amsthm,mathrsfs}%
\usepackage{exscale}
\usepackage{tikz-cd}
\usepackage{relsize}
\usepackage{amssymb}
\usepackage{hyperref}
\usepackage{url}
\usepackage{tikz-cd}
\usepackage{comment}
\usepackage[misc,geometry]{ifsym} 
\usepackage{pst-solides3d}
\newcommand{\bs}{{\scriptscriptstyle \bullet}}

\allowdisplaybreaks

\hypersetup{hypertex=green,
            colorlinks=true,
            linkcolor=green,
            anchorcolor=green,
            citecolor=red}

\theoremstyle{definition}
\newtheorem{theorem}{Theorem}[section]

\newtheorem{remark}{Remark}[section]
\newtheorem{lemma}{Lemma}[section]
\newtheorem{corollary}{Corollary}[section]

\numberwithin{equation}{section}%
\numberwithin{table}{section}%
\numberwithin{figure}{section}

\def\3bar{{|\hspace{-.02in}|\hspace{-.02in}|}}

\newcommand\grad{\operatorname{grad}}
\renewcommand\div{\operatorname{div}}

\newcommand\curl{\operatorname{curl}}
\newcommand\rot{\operatorname{rot}}

\def\d{\text{d}}

\begin{document}
\title[A family of finite element Stokes complexes in 3D]{A family of finite element Stokes complexes in three dimensions}


\author{Kaibo Hu}
\email{khu@umn.edu}
\address{School of Mathematics,
University of Minnesota, 
Minneapolis, MN 55455,
USA.}

\author{Qian Zhang}
\email{go9563@wayne.edu}
\address{Corresponding author. Department of Mathematics, Wayne State University, Detroit, MI 48202, USA. }

\author{Zhimin Zhang}
\email{zmzhang@csrc.ac.cn; zzhang@math.wayne.edu}
\address{Beijing Computational Science Research Center, Beijing, China; Department of Mathematics, Wayne State University, Detroit, MI 48202, USA}
\thanks{This work is supported in part by the National Natural Science Foundation of China grants NSFC 11871092 and NSAF U1930402.}


\date{\today}
\begin{abstract} 
We construct  finite element Stokes complexes on tetrahedral meshes.  In the lowest order case, the finite elements in the complex have 4, 18, 16, and 1 degrees of freedom on each tetrahedron, respectively.  As a consequence, we obtain $\grad\curl$-conforming finite elements and inf-sup stable Stokes pairs on tetrahedral meshes which fit into complexes. We show that the new elements lead to convergent algorithms for solving a $\grad\curl$ model problem as well as solving the Stokes system with precise divergence-free condition. As a by-product, we obtain some nonconforming elements for the $\grad\curl$ model problem. We demonstrate the validity of the nonconforming elements by numerical experiments. 
\end{abstract}	
\maketitle
\section{Introduction}
The discrete de Rham complexes are now an important tool in designing finite elements and analyzing numerical schemes, c.f., \cite{arnold2018finite, arnold2010finite, arnold2006finite, hiptmair1999canonical, neilan2015discrete,christiansen2018nodal}. Motivated by problems in fluid and solid mechanics, there is an increased interest in  de Rham complexes with enhanced smoothness, sometimes referred to as Stokes complexes \cite{tai2006discrete,falk2013stokes,christiansen2016generalized}:
\begin{equation}\label{original-stokes}
\begin{tikzcd}
\!0 \arrow{r} &\mathbb{R} \arrow{r}{\subset} &\! H^{2}(\Omega) \! \arrow{r}{\nabla} & \! \bm H^1(\curl; \Omega)  \arrow{r}{\nabla\times} & \bm H^1(\Omega)\arrow{r}{\nabla\cdot} &L^2(\Omega)  \arrow{r}&0,
 \end{tikzcd}
\end{equation}
where $\bm H^1(\curl; \Omega):=\{\bm{u}\in \bm{H}^{1}(\Omega) : \curl \bm{u}\in \bm{H}^{1}(\Omega)\}$. A slightly different version is the following \cite{evans2013isogeometric,neilan2015discrete}:
\begin{equation}\label{3D:quad-curl}
\begin{tikzcd}
\!0 \arrow{r} &\mathbb{R} \arrow{r}{\subset} &\! H^{1}(\Omega) \! \arrow{r}{\nabla} & \! H(\grad \curl; \Omega)  \arrow{r}{\nabla\times} & \bm H^1(\Omega)\arrow{r}{\nabla\cdot} &L^2(\Omega)  \arrow{r}&0.\!
 \end{tikzcd}
\end{equation} 
Here $H(\grad \curl; \Omega):=\{\bm{u}\in \bm{L}^{2}(\Omega) : \curl \bm{u}\in \bm{H}^{1}(\Omega)\}$ is larger than  $\bm H^1(\curl; \Omega)$ in \eqref{original-stokes}, whereas the last two spaces stay the same.

 Neilan \cite{neilan2015discrete} constructed the first discrete finite element subcomplex of \eqref{original-stokes} on tetrahedral meshes, which involves supersmoothness on lower-dimensional simplices of the mesh. As a result, the construction in \cite{neilan2015discrete} also requires high order polynomials, with degree 9, 8, 7, and 6,  respectively, for the finite elements in the sequence. To reduce the polynomial degree, two discrete complexes are constructed on Alfeld splits \cite{fu2020exact} and Worsey-Farin splits \cite{guzman2020exact}, respectively. As a summary, these constructions involve either a large number of degrees of freedom (DOFs) or an extensive use of macroelement structures.

In this paper we construct a simple discrete subcomplex of \eqref{3D:quad-curl}:
\begin{equation}\label{discrete-complex}
\begin{tikzcd}
0 \arrow{r} &\mathbb{R} \arrow{r}{\subset} & \Sigma_h  \arrow{r}{\nabla} & V_{h}  \arrow{r}{\nabla\times} & \bm \Sigma^{+}_h\arrow{r}{\nabla\cdot} &W_h \arrow{r}&0.
 \end{tikzcd}
\end{equation}
In the lowest order case, the  spaces in \eqref{discrete-complex} have 4, 18, 16, and 1 DOFs on each element, respectively. The DOFs are those of the Whitney forms (e.g., the N\'ed\'elec element and the Raviart-Thomas element), plus vertex evaluation for $V_{h}$ and $\bm\Sigma^{+}_h$. See Figure \ref{fig:third-family} below.
Our construction is inspired by the modified Bernardi-Raugel bubble functions by Guzm\'an and Neilan \cite{guzman2018inf}. For the velocity space $\bm \Sigma^{+}_h$, we extend the low order construction in \cite{guzman2018inf} by enriching the vector-valued Lagrange finite elements with the modified Bernardi-Raugel bubbles and/or suitable interior bubbles in high order cases. Then we construct the entire complex using the Poincar\'e operators.  
The restriction of \eqref{discrete-complex} to each face coincides with the 2D sequences in \cite{HZZcurlcurl2D}.

Applications of \eqref{discrete-complex} include the discretization of incompressible flows, where $\bm \Sigma_h^+$ and $W_h$ are the finite element spaces for the velocity and pressure, respectively. 
Stokes complexes provide a solution to the important problem of preserving the divergence-free condition (incompressibility) precisely in the discretization of the Navier-Stokes equations \cite{falk2013stokes,john2017divergence}.  In this direction, the last two spaces $\bm \Sigma^{+}_h$-$W_{h}$ in \eqref{discrete-complex} extend the low order construction in Guzm\'an and Neilan \cite{guzman2018inf} to an arbitrary order, avoiding supersmoothness or an extensive use of macroelement structures. 
 Moreover, obtaining the entire complex \eqref{discrete-complex} has other benefits. For example, a discrete subcomplex provides an explicit characterization for the kernel of differential operators, which
   is crucial for the construction of robust preconditioners in the framework of the subspace
correction methods and auxiliary space preconditioning technology \cite{lee2007robust,schoeberlthesis,hiptmair2007nodal,xu2009optimal}. With an explicit characterization of the kernel spaces in a discrete complex, one may construct parameter robust preconditioners for solving the Navier-Stokes equations, c.f., \cite{farrell2020reynolds}. 

Another application of \eqref{discrete-complex}, though less addressed in the literature, is on the high order $\curl$ problems in electromagnetism and continuum mechanics \cite{mindlin1962effects,park2008variational,chacon2007steady}. Conforming discretization of $H(\grad\curl; \Omega)$ can be viewed as a natural candidate for solving these problems. In this paper, we show that $V_{h}$ in  \eqref{discrete-complex} lead to a convergence scheme for a high order $\curl$ problem. Our construction thus extends the results in two space dimensions (2D) \cite{HZZcurlcurl2D,WZZelement} and simplifies the 3D $H(\grad\curl)$-conforming element \cite{3D-curlcurlconforming} by two of the authors, which has at least 315 DOFs on each tetrahedron. In fact, solving high order $\curl$ problems is a subtle issue. Similar to the discretization of the Maxwell equations, notorious spurious numerical solutions may appear on non-convex domains if the smoothness of the finite elements are higher than necessary. We leave detailed discussions to future work, but only mention that $V_{h}$ and the entire complex \eqref{discrete-complex} provide the structures that guarantee the convergence of these problems.

The remaining part of the paper is organized as follows. In Section 2, we present preliminaries.
In Section 3, we construct local shape function spaces using various versions of bubble functions and the Poincar\'e operators. 
In Section 4, we define DOFs to construct global discrete Stokes complexes. In Section 5, we prove properties of the discrete complexes, including the exactness and the approximation properties. In Section 6, theoretical analysis is conducted for the $\grad\curl$-conforming elements applying to a high order $\curl$ problem, and numerical experiments are presented to validate the  nonconforming elements.
Finally, we summarize our results and give possible extensions in Section 7.

\section{Preliminaries}

Unless otherwise specified, we assume that $\Omega\in\mathbb{R}^3$ is a contractible Lipschitz domain throughout the paper.  We adopt conventional notations for Sobolev spaces such as $H^m(D)$ or $H^m_0(D)$ on a sub-domain $D\subset\Omega$ furnished with the norm $\left\|\cdot\right\|_{m,D}$ and the semi-norm $\left|\cdot\right|_{m,D}$. In the case of $m=0$, the space $H^{0}(D)$ coincides with $L^2(D)$ which is equipped with the inner product $(\cdot,\cdot)_D$ and the norm $\left\|\cdot\right\|_D$. 
When $D=\Omega$, we drop the subscript $D$. 
We use $\mathring L^2(D)$ to denote $L^2$ functions with vanishing mean:
\[\mathring L^2(D)=\left\{q \in L^2(D): \int_D q \d V=0\right\}.\] We also use  $\bm H^{m}(D)$, $\bm H_0^{m}(D)$,  and ${\bm L}^2(D)$ to denote the vector-valued Sobolev spaces  $[H^{m}(D)]^3$, $[H_0^{m}(D)]^3$, and $[L^2(D)]^3$. 


In addition to the standard Sobolev spaces, we also define
\begin{align*}
	&\quad\quad H(\text{curl};\Omega):=\{\bm u \in {\bm L}^2(\Omega):\; \nabla \times \bm u \in \bm L^2(\Omega)\},\\
	&\quad H(\grad\curl;\Omega):=\{\bm u \in {\bm L}^2(\Omega):\; \nabla \times  \bm u \in \bm H^1(\Omega)\},\\
	&H(\curl^2;\Omega):=\{\bm u \in H(\curl;\Omega):\; \nabla \times \bm u \in H(\curl;\Omega)\}.
\end{align*}

In general $H(\grad\curl;\Omega)\subseteq H(\curl^2;\Omega)$. If $\Omega$ is convex or has $C^{1, 1}$ boundary, then for any function $\bm{u}\in H(\curl^2;\Omega)$ with certain boundary conditions, e.g., $\bm{u}\times \bm{n}=0$ on $\partial \Omega$, we have $\nabla\times \bm{u}\in \bm{H}^{1}$ since $\nabla\times (\nabla\times \bm{u})\in \bm{L}^{2}$ and $\nabla\cdot (\nabla\times \bm{u})=0$ \cite{Girault2012Finite}. This implies that for these domains we actually have $H(\grad\curl;\Omega)\cap H_{0}(\curl;\Omega)=H(\curl^2;\Omega)\cap H_{0}(\curl;\Omega)$, where $H_{0}(\curl;\Omega):=\{\bm u\in H(\curl;\Omega): \bm u\times \bm{n}=0\, \mbox{on } \partial \Omega\}$. In particular, any $H(\curl^2;\Omega)$-conforming finite element is automatically $H(\grad\curl;\Omega)$-conforming. Therefore in this paper we will focus on the construction of $\grad\curl$-conforming finite elements, which naturally fit in the Stokes complex \eqref{3D:quad-curl}.



For a subdomain $D$, we use $P_k(D)$, or simply $P_{k}$ when there is no possible confusion, to denote the space of polynomials with degree at most $k$ on $D$. We also denote by $\mathcal H_k(D)$ the space of homogeneous polynomials of degree $k$ on $D$. 
 Let $\bm P_k=[P_k]^3$ and $\bm {\mathcal H}_k(D)=[\mathcal H_k(D)]^3$ be the corresponding spaces of vector-valued polynomials.  



Let \,$\mathcal{T}_h\,$ be a partition of the domain $\Omega$
consisting of shape-regular tetrahedra. We denote $h_K$ as the diameter of an element $K \in
\mathcal{T}_h$ and $h$ as the mesh size of $\mathcal {T}_h$. Denote by $\mathcal V_h(K)$, $\mathcal E_h(K)$, and $\mathcal F_h(K)$ the sets of vertices, edges, and faces of $K\in \mathcal T_h$.
With the affine mapping
\begin{align}\label{mapping-domain}
F_K(\hat{\bm x})= B_K\hat{\bm x}+ \bm b_K,
\end{align}
we can map the reference element $\hat K$ (the tetrahedron with vertices $(0,0,0)$, $(1,0,0)$, $(0,1,0)$, and $(0,0,1)$) to the element $K$. We use the notation $\hat{\cdot}$ to denote the variables relating to $\hat K$.

For each $K\in \mathcal{T}_h$, let $x_K$ be the barycenter of $K$. We denote $K^r$ as the partition of $K$ by adjoining the vertices of $K$ with the new vertex $x_K$, known as the Alfeld split of $K$ \cite{alfeld1984trivariate}.
We also denote 
\begin{align*}
	{\bm P}_{k}^c(K^r)&=\{\bm v\in \bm  H^1(K): \bm v|_T\in \bm P_k(T) \text{ for all } T\in K^r\},\\
	\mathring {\bm P}^c_{k}(K^r)&=\{\bm v\in \bm  H_0^1(K): \bm v|_T\in \bm P_k(T) \text{ for all } T\in K^r\},\\
	\mathring P_{k}(K^r)&=\{q\in \mathring L^2(K): q|_T\in P_k(T) \text{ for all } T\in K^r\}.
\end{align*}

We use $C$ to denote a generic positive $h$-independent constant. 

\section{Local shape function spaces}

\subsection{Modified bubble functions}

Let $x_1,\cdots,x_4$ be the four vertices of the element $K$, and $x_0=x_K$.
Let $\lambda_0$ be the continuous, piecewise linear function satisfying $\lambda_0(x_j)=\delta_{0j}$ for $0\leq j\leq 4.$
Denote 
\[\bm P_l^{\perp}(K)=\left\{\bm v\in \bm P_l(K):\int_K\bm v\cdot\bm \kappa\d V=0 \text{ for all } \bm \kappa\in \mathcal R_{l-1}\right\}\] 
with $\mathcal R_l=\bm P_{l-1}(K)\oplus\{\bm p \in \bm{\mathcal H}_l(K):\bm x \cdot\bm p =0\}$  for $l\geq 1$ and $\mathcal R_l=0$ for $l=0,-1$.
We will enrich the velocity space with bubble functions from the following space:
\begin{equation}\label{def:Mk}
\bm M_k(K^r)=\{\bm v\in \mathring{\bm P}^c_k(K^r):\bm v=\sum_{j=1}^k\lambda_0^j\bm w_{k-j}\text{ with } \bm w_{k-j}\in \bm P_{k-j}^{\perp}(K)\}.
\end{equation}
We recall the following property  
\cite[Theorem 3.3]{guzman2018inf}.

\begin{lemma}\label{constructingv}
Let $k\geq 1$. For any $K\in \mathcal{T}_h$ and for any $p\in \mathring{P}_{k-1}(K^{r}),$ there exists a unique $\bm v\in \bm M_k(K^r)$ satisfying 
\begin{align*}
	\div \bm v= p \text{ on } K.
\end{align*}
\end{lemma}

Let $\lambda_i (i=1,2,3,4)$ be the barycentric coordinates of $K$, i.e., $\lambda_i(x_j)=\delta_{ij}.$ We define the scalar face  bubbles 
\[B_i=\prod_{j=1,j\neq i}^{4}\lambda_j\text{ for } 1\leq i\leq 4\]
and the scalar interior bubble 
\[B_0=\prod_{j=1}^{4}\lambda_j.\]
The Bernardi-Raugel face bubbles are given as 
\[\bm b_{i}^{f}=B_i\bm n_i \text{ for } 1\leq i\leq 4,\]
where $\bm n_i$ is the outward unit normal to $f_i\in\mathcal F_h(K)$.

According to \cite[Proposition 4.2]{guzman2018inf}, we can modify the Bernardi-Raugel face bubbles such that they have constant divergence.
\begin{lemma}\label{facebubble}
	There exists $\bm \beta_{i}^f\in \bm P^c_{3}(K^r)$ such that
	\begin{equation}\label{defining-modifiedBR}
	\bm \beta_{i}^f|_{\partial K}=\bm b_{i}^{f}|_{\partial K}, \quad \nabla\cdot \bm \beta_{i}^f\in P_{0}(K).
	\end{equation}
\end{lemma}
We refer to the functions $\bm \beta_{i}^f\in \bm P^c_{3}(K^r),\, i=1, 2, 3, 4$ which satisfy \eqref{defining-modifiedBR} as the {\it modified Bernardi-Raugel bubbles} on a tetrahedron $K$ (c.f., \cite{guzman2018inf}). Denote
$$
B^{1}:=\mathrm{span}\{\bm \beta_{i}^f, ~ i=1, 2, 3, 4\},
$$
 To construct high order elements, we will use certain interior bubbles.
Denote
\begin{align*}
	\mathcal S_k(K)=
	\begin{cases}
	\mathcal H_k(K),& k=1,\\
	\mathcal H_k(K)\oplus \mathcal H_{k-1}(K), & k\geq 2,
	\end{cases}
\end{align*}
and
 $$
 	\mathring{\mathcal S}_k(K):=\{u-\frac{1}{|K|}\int_K u\d V : u \in \mathcal S_k(K) \}.
 $$
According to Lemma \ref{constructingv}, there exists a unique subspace $\mathring{B}^{k+1}\subset \bm M_{k+1}(K^{r})$ such that $\nabla\cdot \mathring{B}^{k+1}=\mathring{\mathcal S}_k(K)$, and $\dim \mathring{B}^{k+1}=\dim \mathring{\mathcal S}_k(K)$.



 \begin{remark}
With the constructive proof of Lemma \ref{constructingv} (c.f., \cite[Theorem 3.3]{guzman2018inf}), we can obtain explicit forms of the interior bubbles in the implementation. 
 \end{remark}
\begin{lemma}\label{int_bubble_dof}
	For $k\geq 1$, a function $\bm v\in \mathring B^{k+1}$ is uniquely determined by 
	\begin{align}\label{DOF_bubble_interior}
		\int_{K} \bm v \cdot \nabla q\d V \text{ for all } q\in \mathring{\mathcal S}_k(K). 
	\end{align}
\end{lemma}
\begin{proof}
From the construction, $\dim \mathring{B}^{k+1}=\dim \mathring{\mathcal S}_k(K)$.  Suppose that the functionals in \eqref{DOF_bubble_interior} vanish on $\bm v$. It suffices to show $\bm v=0$. Indeed,  we have from integration by parts
	\[0=\int_K \bm v\cdot \nabla q\d V=\int_K \nabla\cdot \bm v q\d V.\]
	 Taking $q=\nabla\cdot \bm v$, we obtain $\nabla\cdot \bm v=0$ and therefore $\bm v=0$ since $\div: \mathring{B}^{k+1}\to \mathring{\mathcal S}_k(K)$ is bijective by the construction of $\mathring{B}^{k+1}$.
\end{proof}

\subsection{Poincar\'{e} operators}

For any complex 
\begin{equation}\label{general-complex}
\begin{tikzcd}
\cdots \arrow{r}{} & V^{k-1}  \arrow{r}{d^{k-1}}  &V^{k} \arrow{r}{d^k} &  V^{k+1} \arrow{r}{}& \cdots,
 \end{tikzcd}
\end{equation}
where $V^{\bs}$ are linear vector spaces and $d^{\bs}$ are linear operator, we call a graded operator $\mathfrak p^k: V^k\rightarrow V^{k-1}$ Poincar\'{e} operators if it satisfies
\begin{itemize}
	\item the null-homotopy property: 
	   \begin{align}\label{null-homotopy}
	   	d^{k-1}\mathfrak p^k+\mathfrak p^{k+1}d^{k}=\operatorname{id}_{V^k};
	   \end{align}
	\item the complex property:
	   \begin{align}\label{complexproperty}
	   	\mathfrak p^{k-1}\circ \mathfrak p^k=0.
	   \end{align}
	   \end{itemize}
	   \begin{lemma}\label{lem:exactness}
	   If there exist Poincar\'e operators $\mathfrak{p}^{\bs}$ for \eqref{general-complex}, then \eqref{general-complex} is exact.
	   \end{lemma}
	   \begin{proof}
	Assume that $d^{k}u=0$ for $u\in V^{k}$.   From the null-homotopy identity, $u=d^{k-1}(\mathfrak{p}^{k}u)$. This implies the exactness of \eqref{general-complex} at $V^{k}$.
	   \end{proof}
	   For the 3D de Rham complex,
\begin{equation}\label{general-complex-smooth}
\begin{tikzcd}
0 \arrow{r}{} &\mathbb{R} \arrow{r}{\subset} & C^{\infty}  \arrow{r}{d^{0}}  &\left[C^{\infty}\right]^{3}\arrow{r}{d^1} & \left[C^{\infty}\right]^{3}\arrow{r}{d^{2}}& C^{\infty} \arrow{r}{}&0,
 \end{tikzcd}
\end{equation}	  
there exist Poincar\'{e} operators and they have the	 explicit form \cite{hiptmair1999canonical, lang2012fundamentals,christiansen2016generalized}:
\begin{align}\label{p-1}
	\mathfrak p^1 \bm u&=\int_0^1 \bm u(W+t(\bm x-W))\cdot (\bm x-W)\d t, \\\label{p-2}
	\mathfrak p^2 \bm u&=\int_0^1 \bm u(W+t(\bm x-W))\times t(\bm x-W)\d t, \\\label{p-3}
	\mathfrak p^3 u&=\int_0^1 t^2 u(W+t(\bm x-W))(\bm x-W)\d t,
\end{align}
where $W$ is a base point. In addition to the complex property and the null-homotopy identity, these operators further satisfy   
	   \begin{itemize}
	\item  the polynomial preserving property:
	   if $u$ is a polynomial of degree $r$, then $\mathfrak p u$ is a polynomial of degree at most $r+1$.
\end{itemize}

\subsection{Local shape function spaces}
On each $K\in \mathcal{T}_h$, we construct the local shape function spaces of  \eqref{discrete-complex} as follows:
\begin{equation}\label{local-complex}
\begin{tikzcd}
\!0 \!\arrow{r}\! &\mathbb{R} \!\arrow{r}{\subset} \!& \Sigma^{r}_h(K) \! \arrow{r}{\nabla} \!&V^{r-1,k+1}_{h}(K)\! \arrow{r}{\nabla\times} \!&  \bm \Sigma_h^{k,+}(K)\! \arrow{r}{\nabla\cdot} \!& W_h^{k-1}(K)\!\arrow{r}\!&\!0.
 \end{tikzcd}
\end{equation}
Different choices of $r$ and $k$ will lead to various versions of the complex \eqref{local-complex}.

We choose $\Sigma^{r}_h(K):=P_{r}(K)$, $W^{k-1}_h(K):=P_{k-1}(K)$, and set $\bm \Sigma_h^{k,+}(K)=\bm P_{k}(K)\oplus B$,
	where
	$$
B=
\begin{cases}
	B^1,& k=1,\\
	B^1\oplus\mathring B^2,& k=2,\\
	\mathring B^{k},& k\geq 3.
	\end{cases}
	$$
Note that for $k=1$, we only supply $\bm P_1(K)$ with the modified Bernardi-Raugel face bubbles; for $k=2$, we supply $\bm P_{2}(K)$ with both face and interior bubbles, while for $k\geq 3$ we only need to supply $\bm P_k(K)$ with interior bubbles.
It is easy to see the face bubbles $\{\bm \beta_i^f\}_{i=1}^4$ and $\bm P_2(K)$ are linearly independent, and hence, $\bm P_2(K)\oplus B^1$ and $\bm P_1(K)\oplus B^1$ are direct sums. From the explicit form \eqref{def:Mk} of the functions in $\boldsymbol{M}_{k}\left(K^{r}\right)$, we see that $\boldsymbol{M}_{k}\left(K^{r}\right) \oplus \bm{P}_{k}(K)$ is a direct sum, and hence, $\mathring B^{k} \oplus \bm{P}_{k}(K),$ is also a direct sum. 

\begin{remark}
	The idea of enriching with modified bubbles is inspired by \cite{guzman2018inf}, where the case of $k=1$ is defined and used to construct a stable Stokes finite element pair. Here we extend it to high order cases.
\end{remark}

Define
\begin{align}\label{Vh2}
	V_h^{r-1,k+1}(K)=\nabla \Sigma^{r}_h(K)\oplus \mathfrak{p}^2\bm \Sigma^{k,+}_h(K).
\end{align}
The right-hand side of \eqref{Vh2} is a direct sum. In fact, if $\bm u\in \nabla\Sigma^{r}_h(K)\cap \mathfrak{p}^2\bm \Sigma^{k,+}_h(K)$, then $\nabla\times\bm u=0$ and $\mathfrak p^1\bm u=0$. By the null-homotopy identity \eqref{null-homotopy}, $\bm u=\mathfrak p^2\nabla\times\bm u+\nabla\mathfrak p^1\bm u=0.$
\begin{remark}
	For the bubble functions in $\bm \Sigma_h^{k,+}(K)$, we choose the barycenter $x_K$ as the base point $W$, c.f., \cite{christiansen2016generalized}. For other functions, we choose $W=0$ to be the origin.
\end{remark}
\begin{remark}
	The Koszul operator $\kappa \bm u:=\bm u\times \bm x$ exerting on homogeneous polynomials has similar properties as the Poincar\'e  operator $\mathfrak p^2$ \cite{arnold2018finite, arnold2006finite}. For polynomial bases in $\bm \Sigma_h^{k,+}(K)$ other than the bubbles, we can replace the Poincar\'e operator $\mathfrak p^2$ by the Koszul operator. However, to get the complex property it seems necessary to use the Poincar\'{e} operators for the bubbles. 
	\end{remark}

When $r=k$ with $k\geq 1$ and $r=k+1,k+2$ with $k=1,2$, the tangential components of functions in $V_h^{r-1,k+1}(K)$ on $\partial K$ may not be polynomials of order $r-1$. This will render the elements in these cases nonconforming.
To make them conforming, for $\mathfrak p^2\bm w\in V_h^{r-1,k+1}(K)$, we shall subtract a high order polynomial  such that the resulting function  has low order tangential components on $\partial K$. The high order polynomial should be curl-free so that it will not affect the complex property and exactness. 

We will construct the correction term in the space
\begin{align*}
	R_k(K):=&\nabla P_{k+1}(K)\oplus \bm P_0(K)\times \bm x.
\end{align*}

We first present the DOFs to determine a polynomial in $R_k(K)$.
\begin{lemma}\label{DOFsofRk} For $k\geq 1$, the following DOFs for $\bm u\in R_k(K)$ are unisolvent and lead to a conforming subspace in $H(\curl;\Omega)$:
	\begin{align}
	&\int_{e_i}\bm u\cdot \bm \tau_iq\d s\text{ for all }q\in P_{k}(e_i),\ i=1,2,\cdots, 6,\label{dofset1}\\
    &\int_{f_i}\bm u\cdot \bm q\d A \text{ for all }\bm q\in P_{k-2}(f_i)(\bm x-(\bm x\cdot\bm n_i)\bm n_i),\ i=1,2,3,4,\label{dofset2}\\
    &\int_K \bm u\cdot \bm q\d V\text{ for all }\bm q\in P_{k-3}(K)\bm x\label{dofset3}.
\end{align}
\end{lemma}
\begin{remark}
For $k=0$, the space $R_0(K)$ and the DOFs define the lowest order N\'ed\'elec element. 
\end{remark}
\begin{proof}
We first prove the conformity. We assume that the DOFs \eqref{dofset1}--\eqref{dofset2} vanish on $\bm u\in R_k(K)$ and prove $\bm u\times \bm n_i=0$ on face $f_i$.
	Since $\bm u\in R_k(K)$, $\nabla\times\bm u\in D_0(K):=\bm P_0(K)\oplus P_0(K)\bm x$ (the lowest-order Raviart-Thomas element). By the Stokes theorem,
	\[\int_{f_i} \nabla\times\bm u\cdot\bm n_i \d A=\int_{\partial f_i}\bm u\cdot\bm \tau ds=0,\]
	which implies $\nabla\times\bm u=0$. It then follows that $\bm u=\nabla p$ with $p\in P_{k+1}(K).$
	From the vanishing DOFs \eqref{dofset1}, the directional derivative of $p$ along $\partial f_i$ is 0. Consequently, we can choose $p$ such that $p|_{f_i}=\lambda_{i,1}\lambda_{i,2}\lambda_{i,3}\theta_i$ with $\theta_i\in P_{k-2}(f_i)$ and barycentric coordinates $\lambda_{i,1}\lambda_{i,2}\lambda_{i,3}$ of the face $f_i$. By integration by parts,
	\begin{align*}
	0=\int_{f_i}\bm u\cdot \bm q\d A=\int_{f_i}\nabla_{f_i}p\cdot \bm q\d A&=\int_{\partial f_i}\lambda_{i,1}\lambda_{i,2}\lambda_{i,3}\theta_i\nabla_{f_i}\cdot\bm q\d s, \\
	&\bm q\in P_{k-2}(f_i)(\bm x-(\bm x\cdot\bm n_i)\bm n_i).
	\end{align*}
	Choosing $\bm q$ such that $\nabla_{f_i}\cdot\bm q=\theta_i$ leads to $\theta_i=0$, and hence $p|_{f_i}=0$. Therefore, $\bm u\times\bm n_i=\nabla_{f_i}p\times\bm n_i=0.$
	
We then prove the unisolvence. The dimension of $R_k(K)$ coincides with the number of DOFs \eqref{dofset1}--\eqref{dofset3}. It suffices to show $\bm u=0$ if all the DOFs vanish on $\bm u\in R_k(K)$.	
	Since $p|_{f_i}=0$, we have 	
$p=\lambda_1\lambda_2\lambda_3\lambda_4 \varphi$ with $\varphi\in P_{k-3}(K)$. 
Using the DOFs \eqref{dofset3} and proceeding as the proof of $\theta_i=0$, we have $\varphi=0$, and then $\bm u=0$.
\end{proof}

\begin{remark}\label{curluvanish}
	From the above proof, we can see if $\int_{e_i}\bm u\cdot \bm \tau_i\d s=0$, $i=1,2,\cdots,6$ for $\bm u\in R_k(K)$, then $\nabla\times\bm u=0$.
\end{remark}

 In the following, we will construct the correction term by specifying the DOFs \eqref{dofset1}--\eqref{dofset3}. To this end, we first modify a 2D function such that its tangential component on each edge vanishes and its $\rot$ only differs by a constant compared with the original $\rot$. 
Denote 
\begin{align*}
	R_0(f)=&\nabla_f P_{1}(f)\oplus P_0(f) \bm x_f^{\perp}\ (\text{the lowest-order Ned\'el\'ec element space in 2D}),\\
	\mathfrak p_f w=&\int_{0}^{1}t \bm{x}_f^{\perp}w(t\bm x_f)\d t\ (\text{the 2D Poincar\'e operator}),
\end{align*}
where  $\bm{x}_f^{\perp}:=(-x_{2}, x_{1})^T$ for $\bm{x}_f:=(x_{1}, x_{2})^T$.
The 2D Poincar\'e operator $\mathfrak p_f$ satisfies $\rot \mathfrak p_f=\text{id}$.

 \begin{lemma}[\cite{hu2020simple-revised}]\label{lemma1}
	For $w\in P_k(f)$ with $k\geq 0$, there exist a mapping $\varphi: P_k(f)\rightarrow P_{k+1}(f)$ and a function $\bm r_w\in R_0(f)$ such that $\widetilde{\mathfrak p}_f w:=\mathfrak p_f w-\nabla_f \varphi(w)-\bm r_w$ has vanishing tangential component on each edge $e$ of $f$ and $\rot \widetilde{\mathfrak p}_f w=\rot \mathfrak p_f w-\rot \bm r_w\in P_0(f)$ with $\rot \bm r_w\in P_0(f)$.
\end{lemma}

We are now in a position to construct the correction term. 
\begin{lemma}\label{lemma2}
	For $\bm w\in \bm\Sigma_h^{k,+}(K)$ with $k\geq 1$, there exists a function $\bm \psi_{\bm w}\in R_m(K)$ with $m=\max\{k+1,4\}$ such that $\nabla\times \bm \psi_{\bm w}=0$ and if $\nabla\times(\mathfrak p^2 \bm w- \bm \psi_{\bm w})\cdot \bm n_{i}=0$, then $\bm n_i\times(\mathfrak p^2 \bm w- \bm \psi_{\bm w})\times\bm n_i$ belongs to $R_0(f_i)$ on the faces $f_i,\ i=1,2,3,4$ of $K$.
	If  $\mathfrak p^2 \bm w=0$, then $\bm \psi_{\bm w}=0$.
\end{lemma}
\begin{proof}
We first construct a function $\bm \gamma_{\bm w}\in R_0(K)$ such that
\[\int_{e_i}\bm \gamma_{\bm w}\cdot\bm \tau_i\d s=\int_{e_i}\mathfrak p^2 \bm w\cdot\bm \tau_i\d s.\]
Denote $\omega_i= \nabla\times\mathfrak p^2 \bm w\cdot \bm n_{i},$ and define
\begin{align}\label{tildew}
	\widetilde{\bm w}_{f_i}:=\bm n_i\times (\mathfrak p^2 \bm w-\bm \gamma_{\bm w})\times \bm n_i-\widetilde{\mathfrak p}_{f_i} \omega_i,
\end{align}
 	where $\widetilde{\mathfrak p}_{f_i}$ is defined in Lemma \ref{lemma1}.
 The function  $\widetilde{\bm w}_{f_i}$ satisfies
\begin{align*}
&\widetilde{\bm w}_{f_i}\in [P_{m}(f_i)]^2,\\
&\nabla_{f_i}\times \widetilde{\bm w}_{f_i}=\nabla_{f_i}\times( \bm r_{\omega_i}-\bm n_i\times\bm \gamma_{\bm w}\times \bm n_i)\in P_{0}(f_i),\\
&\widetilde{\bm w}_{f_i}\cdot \bm \tau_{\partial f_i} = (\mathfrak p^2 \bm w-\bm\gamma_{\bm w})\cdot \bm \tau_{\partial f_i}\in P_m,\\
&\int_{\partial f_i}\widetilde{\bm w}_{f_i}\cdot \bm \tau_{\partial f_i}\d s=0.
\end{align*}
We construct a function $\bm \psi_{\bm w}\in R_m(K)$ by setting 
\begin{align*}
	&\int_{e_i}\bm \psi_{\bm w}\cdot \bm \tau_iq\d s=\int_{e_i}(\mathfrak p^2 \bm w-\bm\gamma_{\bm w})\cdot \bm \tau_iq\d s \text{ for all }q\in P_{m}(e_i),\ i=1,2,\cdots, 6,\\
    &\int_{f_i}\bm \psi_{\bm w}\cdot \bm q\d A=\int_{f_i}\widetilde{\bm w}_{f_i}\cdot \bm q\d A \text{ for all }\bm q\in P_{m-2}(f_i)(\bm x-(\bm x\cdot\bm n_i)\bm n_i),\ i=1,2,3,4,\\
    &\int_K \bm \psi_{\bm w}\cdot \bm q\d V= 0\text{ for all }\bm q\in P_{m-3}(K)\bm x.
\end{align*}
From Lemma \ref{DOFsofRk} and Remark \ref{curluvanish}, we have $\bm n_i\times\bm \psi_{\bm w}\times\bm n_i=\widetilde{\bm w}_{f_i}$ on $f_i$ and $\nabla\times \bm \psi_{\bm w}=0$. Therefore from \eqref{tildew}, $\bm n_i\times(\mathfrak p^2 \bm w- \bm \psi_{\bm w})\times\bm n_i=\bm n_i\times\bm \gamma_{\bm w}\times\bm n_i+\widetilde{\mathfrak p}_{f_i} \omega_i$ and
$\nabla\times(\mathfrak p^2 \bm w- \bm \psi_{\bm w})\cdot \bm n_{i}= \nabla_{f_i}\times[\bm n_i\times(\mathfrak p^2 \bm w- \bm \psi_{\bm w})\times\bm n_i] =\nabla_{f_i}\times (\bm n_i\times\bm \gamma_{\bm w}\times\bm n_i+\widetilde{\mathfrak p}_{f_i} \omega_i)=\omega_i+\nabla_{f_i}\times(\bm n_i\times\bm \gamma_{\bm w}\times\bm n_i-\bm r_{\omega_i})$ on $f_i$.

If $\nabla\times(\mathfrak p^2 \bm w- \bm \psi_{\bm w})\cdot \bm n_{i}=0$, then 
\[\omega_i+\nabla_{f_i}\times(\bm n_i\times\bm \gamma_{\bm w}\times\bm n_i-\bm r_{\omega_i})=0,\]
which implies $\omega_i=-\nabla_{f_i}\times(\bm n_i\times\bm \gamma_{\bm w}\times\bm n_i-\bm r_{\omega_i})\in P_0(f_i)$, 
and hence, ${\mathfrak p}_{f_i} \omega_i\in R_0(f_i)$ and $\varphi({\omega_i})\in P_1(f_i)$ (the mapping $\varphi$ is defined in Lemma \ref{lemma1}). Therefore 
\begin{align*}
&\bm n_i\times(\mathfrak p^2 \bm w- \bm \psi_{\bm w})\times\bm n_i=\bm n_i\times\bm\gamma_{\bm w}\times\bm n_i+\widetilde{\mathfrak p}_{f_i} \omega_i\\
=&\bm n_i\times\bm\gamma_{\bm w}\times\bm n_i+{\mathfrak p}_{f_i} \omega_i-\nabla_{f_i}\varphi(\omega_i)-\bm r_{\omega_i}\in R_0(f_i).
\end{align*}
\end{proof}

The operator  ${\mathfrak p}^2$ is then modified as
\[\widetilde{\mathfrak p}^2\bm w={\mathfrak p}^2\bm w-\bm \psi_{\bm w},\] 
and the space 
\[V_h^{r-1,k+1}(K)=\nabla \Sigma_h^r(K)\oplus \widetilde{\mathfrak p}^2\bm \Sigma_h^{k,+}(K),\]
when $r=k$ with $k\geq 1$ and $r=k+1,k+2$ with $k=1,2$.



\begin{lemma}\label{local_exactness}
The local sequence \eqref{local-complex} is an exact complex. 
\end{lemma}
\begin{proof}
	By the definition of the shape function spaces, it is easy to show that the sequence \eqref{local-complex} is a complex. It remains to show the exactness. 
	We only show the exactness at $V_h^{r-1,k+1}(K)$. To this end, we show that, for any $\bm v_h\in V^{r-1, k}_h(K)$ for which $\nabla\times\bm v_h=0$, there exists a $p_h\in \Sigma^{r}_h(K)$ s.t. $\bm v_h=\nabla p_h.$
Since $\bm v_h\in V^{r-1, k}_h(K)$, we have $\bm v_h=\nabla p_h+\mathfrak p^2 \bm w_h$ or $\bm v_h=\nabla p_h+\widetilde{\mathfrak p}^2\bm  w_h$ with $p_h\in \Sigma^{r}_h(K)$ and $\bm w_h\in \bm \Sigma^{k,+}_h(K)$. By the null-homotopy identity \eqref{null-homotopy} and the fact that $\nabla\times\bm\psi_{\bm w_h}=0$, $0=\nabla\times\bm v_h=\nabla\times\mathfrak p^2 \bm w_h=\bm w_h-\mathfrak p^3\nabla\cdot\bm w_h$, which leads to $\bm w_h=\mathfrak p^3\nabla\cdot\bm w_h$. By the complex property \eqref{complexproperty}, $\mathfrak p^2 \bm w_h=0$ ($\widetilde{\mathfrak p}^2 \bm w_h=0$ since $\bm\psi_{\mathfrak p^3\nabla\cdot\bm w_h}=0$). 
\end{proof}

From the definition, we see that $V_{h}^{r-1, k+1}(K)$ has two parts: one from the gradient on $\Sigma^{r}_h(K)$ and the other from the Poincar\'e operator on $\bm \Sigma_{h}^{k,+}$. The first part is easy to implement: we may remove the constant (kernel of gradient) from  the bases of  $\Sigma_{h}^{r}$ and apply gradient to the rest. The $\mathfrak{p}^2\bm \Sigma^{k,+}_h(K)$ part calls for more explanation as we cannot obtain a basis by applying the Poincar\'e operator to a basis of $\bm \Sigma_{h}^{k,+}$ (as the results are not linearly independent).  Now we show how to obtain a basis for the $\mathfrak{p}^2\bm \Sigma^{k,+}_h(K)$ part to implement $V_h^{r-1,k+1}(K)$. 

We first claim $\bm P_{k}(K)=\nabla\times \bm P_{k+1}(K)\oplus \mathfrak p^3 P_{k-1}(K)$. In fact, for all $\bm u\in \bm P_{k}(K)$, the null-homotopy identity \eqref{null-homotopy} leads to $\bm u=\nabla\times\mathfrak{p}^2\bm u+\mathfrak{p}^3\nabla\cdot\bm u\in\nabla\times \bm P_{k+1}(K)+ \mathfrak p^3 P_{k-1}(K)$. Moreover, if $\bm u\in \nabla\times \bm P_{k+1}(K)\cap \mathfrak p^3 P_{k-1}(K)$, then $\nabla\cdot\bm u=0$ and $\mathfrak p^2\bm u=0$, which follows from \eqref{null-homotopy} again that $\bm u=0$. 

We then have the decomposition $\bm \Sigma_h^{k,+}(K)=\nabla\times \bm P_{k+1}(K)\oplus \mathfrak p^3 P_{k-1}(K)\oplus B$, which leads to
\begin{align}\label{decomp-2}
\mathfrak{p}^{2}\bm \Sigma^{k,+}_{h}(K)=\mathfrak{p}^{2}\nabla\times\bm{P}_{k+1}(K)+\mathfrak{p}^{2}B+\mathfrak{p}^{2}\mathfrak{p}^{3}{P}_{k-1}(K)=\mathfrak{p}^{2}\nabla\times\bm {P}_{k+1}+\mathfrak{p}^{2}B,
\end{align}
where we used $\mathfrak{p}^{2}\mathfrak{p}^{3}=0$.

From the exactness and the decomposition of $\bm \Sigma_h^{k,+}(K)$, we obtain
\begin{align*}
\dim\mathfrak p^2\bm \Sigma_h^{k,+}&=\dim V^{r-1,k+1}_{h}(K)-\dim \nabla \Sigma^{r}_h(K)\nonumber\\
&=\dim \bm \Sigma_h^{k,+}(K)-\dim W_h^{k-1}(K)\\
&=\dim \nabla\times \bm P_{k+1}(K)+\dim \mathfrak p^3 P_{k-1}(K)+\dim B-\dim W_h^{k-1}(K)\\
&=\dim \nabla\times \bm P_{k+1}(K)+\dim B\\
&\geq \dim \mathfrak p^2\nabla\times \bm P_{k+1}(K)+\dim\mathfrak p^2 B,
\end{align*}
which together with \eqref{decomp-2} leads to 
\[\mathfrak p^2\bm \Sigma_h^{k,+}(K)=\mathfrak p^2\nabla\times\bm P_{k+1}(K)\oplus \mathfrak p^2 B.\]

Therefore, to implement $\mathfrak p^2 \bm \Sigma_h^{k,+}(K)$, we take the bases of $B$ and the bases of $\nabla\times \bm P_{k+1}(K)$, and apply the Poincar\'e operator $\mathfrak p^2$.  We then can implement $\widetilde{\mathfrak p}^2 \bm \Sigma_h^{k,+}(K)$.

To show the approximation property of the finite element space $V_h^{r-1,k+1}$, we demonstrate that $V_h^{r-1,k+1}(K)$ contains polynomials of certain degree.
\begin{lemma}\label{Vh} The inclusion $\bm P_{s}(K)\subseteq V^{r-1, k+1}_h(K)$ holds, where $s=\min\{r-1, k+1\}$. 
\end{lemma}

\begin{proof}
From the null-homotopy property, $\bm P_{s}(K)= \grad\mathfrak{p}^{1}\bm P_{s}(K)+\mathfrak{p}^{2}\curl \bm P_{s}(K)$.
By definition, $V^{r-1, k+1}_h(K)=\grad\Sigma^{r}_h(K)+\mathfrak{p}^{2}\bm \Sigma_h^{k,+}(K)$. For $s=\min\{r-1, k+1\}$, we have $\mathfrak{p}^{1}\bm P_{s}(K)\subseteq P_{r}(K)= \Sigma^{r}_h(K)$ and $\curl \bm P_{s}(K)\subseteq \bm P_{k}(K)\subseteq  \bm \Sigma_h^{k,+}(K)$.  Therefore the desired inclusion holds. Similarly, we can prove the lemma for the case when $\widetilde{\mathfrak{p}}^2$ is involved.

\end{proof}


\section{Degrees of freedom and global finite element spaces}\label{sec:tetrahedral}

In this section, we construct $\grad\curl$-conforming finite elements and discrete Stokes complexes on tetrahedra.
The discrete complex with global finite element spaces is given by
\begin{equation}\label{discrete-complex-r-k}
\begin{tikzcd}
0 \arrow{r} &\mathbb{R} \arrow{r}{\subset} & \Sigma^{r}_h  \arrow{r}{\nabla} &V^{r-1,k+1}_{h} \arrow{r}{\nabla\times} &  \bm \Sigma_h^{k,+} \arrow{r}{\nabla\cdot} & W_h^{k-1}\arrow{r}&0.
 \end{tikzcd}
\end{equation}
 Taking $r=k$, $k+1$, and $k+2$ in \eqref{discrete-complex-r-k}  yields three versions of $\grad\curl$-conforming element spaces $V_h^{k-1,k+1}$, $V_h^{k,k+1}$, and $V_h^{k+1,k+1}$. Fig. \ref{fig:third-family} demonstrates the complex \eqref{discrete-complex-r-k} for the case $k=1$.
We define DOFs for each space in \eqref{discrete-complex-r-k}.

The DOFs for the Lagrange element $\Sigma^{r}_{h}(K)$ can be given as follows.
\begin{itemize}
\item Vertex DOFs $ M_{v}({u})$ at all the vertices  $v_{i}\in \mathcal V_h(K)$:
$$
M_{v}(u)=\left\{u\left({v}_{i}\right)\right\}.
$$
\item Edge DOFs $M_{e}(u)$ on all the edges $e_{i}\in \mathcal E_h(K)$:
\begin{align*} M_{e}(u)=\left\{ \int_{e_i} u v \mathrm{d} s\text { for all } v \in P_{r-2}(e_i) \right\}.
\end{align*}
 \item Face DOFs $M_{f}(u)$  on all the faces $f_{i}\in \mathcal F_h(K)$:
\begin{align*} M_{f}(u)=\left\{ \int_{f_i} u v \mathrm{d} A\text { for all } v \in P_{r-3}(f_i) \right\}.
\end{align*}
\item Interior DOFs $M_{K}(u)$ in the element $K$: 
$$M_{K}(u)=\left\{\int_{K} u v \mathrm{d} V \text{ for all } v \in P_{r-4}(K)
\right\}.$$
\end{itemize}

\begin{figure}
\includegraphics[width=\textwidth]{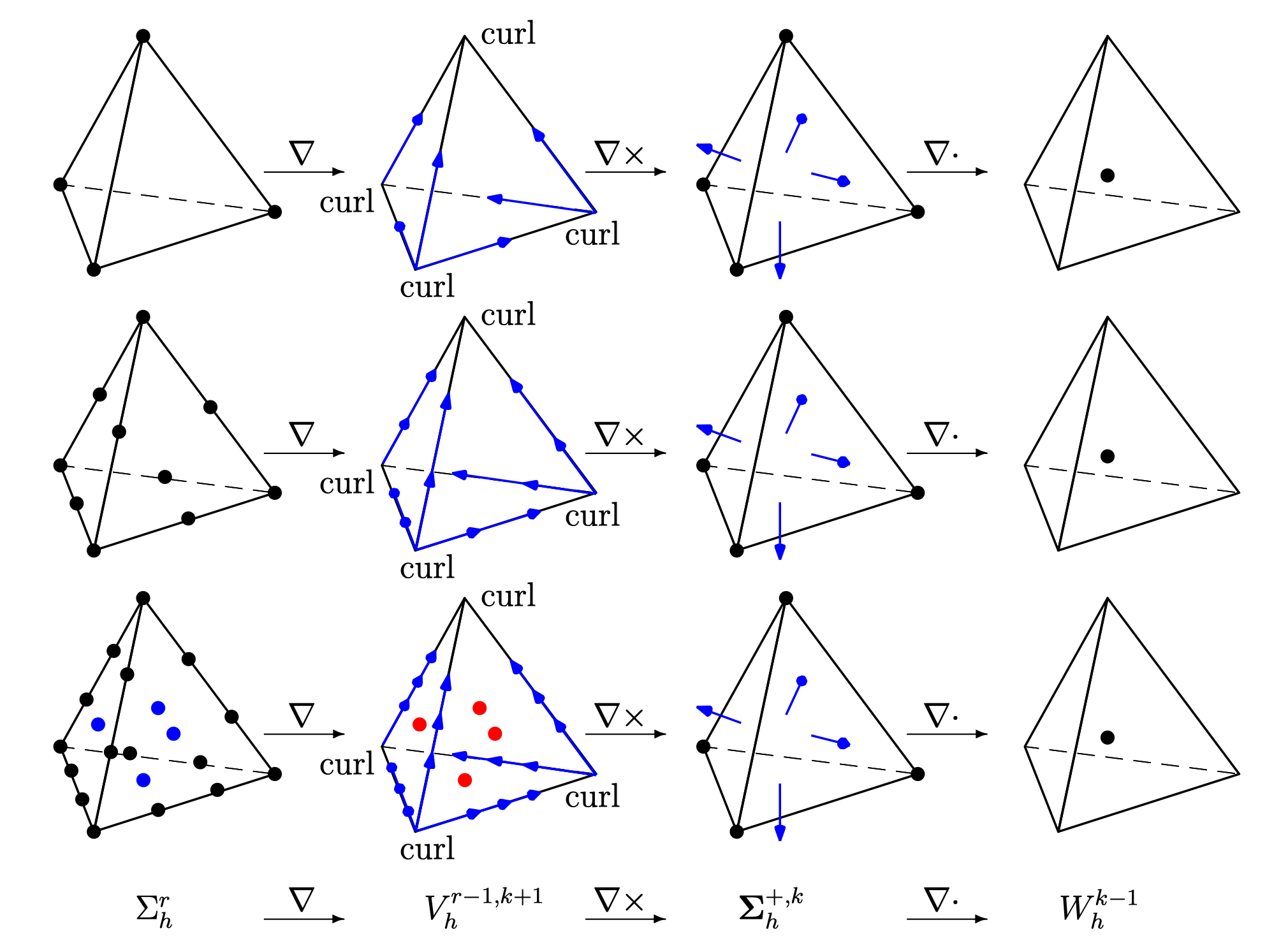}
\caption{The lowest-order ($k=1$) finite element complex \eqref{discrete-complex-r-k} on tetrahedra with $r=k$ in the first row, $r=k+1$ in the second row, and $r=k+2$ in the third row.}
\label{fig:third-family}
\end{figure}
We now equip the space $V^{r-1,k+1}_h(K)$ with the following DOFs:
\begin{itemize}
		\item Vertex DOFs $\bm M_{ {v}}({\bm u})$ at all  vertices $ {v}_{i}\in \mathcal V_h(K)$:
	\begin{equation}\label{tridef1-1}
	\bm M_{ {v}}({\bm u})=\left\{( \nabla\times {\bm u})(v_{i})\right\}.
	\end{equation}
	\item Edge DOFs $\bm M_{ {e}}({\bm u})$ on all edges $ {e}_i\in \mathcal E_h(K)$:
	\begin{align}
		 &\bm M_{ {e}}(  {\bm u})=	\left\{\int_{e_i}{\bm u}\cdot\bm \tau_iq\d s\text{ for all }{q}\in P_{r-1}( {e}_i)\right\}\nonumber\\
		 &\cup \left\{\int_{e_i}\nabla\times{\bm u}\cdot\bm q\d s\text{ for all }  \bm{q}\circ F_K\in \bm P_{k-2}(\hat{e}_i)\right\}.	\label{tridef1-2}\end{align}	
	\item Face DOFs $\bm M_{f}({\bm u})$ at all faces $ {f}_i\in \mathcal F_h(K)$ (with two mutually orthogonal unit vector $\bm \tau_i^1$ and $\bm\tau_i^2$ in the face $f_i$ and the unit  normal vector ${\bm n}_i$):
	\begin{align}
		 &\bm M_{f}({\bm u})=\left\{\int_{f_i}\nabla\times{\bm u}\cdot \bm n_i q\d A\text{ for all }{q}\in P_{k-3}( {f}_i)\slash \mathbb{R}\right\}\nonumber\\
		 &\quad\quad\cup \left\{\int_{f_i}\nabla\times{\bm u}\cdot \bm \tau_i^1 q\d A\text{ for all }{q}\in P_{k-3}( {f}_i)\right\}\nonumber\\
		&\quad\quad\cup  \left\{\int_{f_i}\nabla\times{\bm u}\cdot \bm \tau_i^2q\d A\text{ for all }{q}\in P_{k-3}( {f}_i)\right\}\label{tridef1-3}\\
		 &\cup \left\{\int_{f_{i}} \boldsymbol{u} \cdot \boldsymbol{q} \mathrm{d} A\text{ for all }\boldsymbol{q}\circ F_K=B_{K} \hat{\boldsymbol{q}},\hat{\boldsymbol{q}} \in P_{r-3}(\hat{f}_{i})\hat{\bm x}_{\hat f_i}\right\}\nonumber,
	\end{align}	
	where $\hat{\bm x}_{\hat f_i}=\big[\hat{\boldsymbol{x}}-(\hat{\boldsymbol{x}}\cdot\hat{\bm n}_i)\hat{\bm n}_i\big]\big|_{\hat f_i}$.
    \item Interior DOFs $\bm M_{ {K}}(  {\bm u})$ for the element $K$:
	\begin{align}\label{tridef1-4}
	&\quad\bm M_{K}({\bm u})=\left\{\int_{K} {\bm u}\cdot  {\bm q}\mathrm \d V\text{ for all }\bm q\circ F_K=B_K\hat{\bm q}, \hat{\bm q}\in  P_{r-4}(\hat K)\hat{\bm x}\right\}\nonumber\\
	 &\left\{\int_{K}  \nabla\times{\bm u}\cdot  {\bm q}\mathrm \d V\text{ for all }\bm q\circ F_K =B_K^{-T}\hat{\bm q}, \hat{\bm q}\in\hat{\nabla}\times\mathring V_h^{r-1,k+1}(\hat K)\right\},
		\end{align}
where $\mathring V_h^{r-1,k+1}(K)=\{\bm u\in V_h^{r-1,k+1}(K): \text{DOFs \eqref{tridef1-1}--\eqref{tridef1-3} vanish on }\bm u\}.$
		\end{itemize}
		

The DOFs for $\bm \Sigma_{h}^{k,+}(K)$ can be given similarly to $\Sigma_h^r(K)$ with some additional face or interior integration DOFs to take care of the bubble functions (see Lemma \ref{facebubble} and Lemma \ref{int_bubble_dof}). 
\begin{itemize}
\item Vertex DOFs $ \bm M_{v}(\bm {u})$ at all the vertices  $v_{i}\in \mathcal V_h(K)$:
\begin{align} \label{vector-sigma-DOFs1}
\bm M_{v}(\bm u)=\left\{\bm u\left({v}_{i}\right) \right\}.
\end{align}
\item Edge DOFs $\bm M_{e}(\bm u)$ on all the edges $e_{i}\in \mathcal E_h(K)$:
\begin{align} \bm M_{e}(\bm u)=\left\{\int_{e_i} \bm u \cdot \bm v \mathrm{d} s\text { for all } \bm  v \in \bm P_{k-2}(e_i)\right\}.
\end{align}
 \item Face DOFs $\bm M_{f}(\bm u)$ on all the faces $f_{i}\in \mathcal F_h(K)$:
\begin{align} \label{vector-sigma-DOFs3}
\bm M_{f}(\bm u)&=\left\{\int_{f_i} \bm u\cdot  \bm v \mathrm{d} A\text { for all } \bm v \in \bm P_{k-3}(f_i) \right\}\nonumber\\
&\cup\left\{ \int_{f_i} \bm u\cdot\bm n_i  \mathrm{d} A \text{ when } k=1,2\right\}.
\end{align}
\item Interior DOFs $\bm M_{K}(\bm u)$ in  the element $K$: 
\begin{align}\label{vector-sigma-DOFs4}
\bm M_{K}(\bm u)=&\left\{\int_{K} \bm u\cdot  \bm v \mathrm{d} V \text{ for all } \bm v =B_K^{-T}\hat{\bm v}, \hat{\bm v}\in\hat{\nabla}\times\mathring V_h^{r-1,k+1}(\hat K)\right\}\nonumber
	 \\
	&\cup \left\{\int_{K} \bm u \cdot\nabla v \mathrm{d} V \text{ for all } v \in  P_{k-1}(K)\slash \mathbb R\right\}.
	 \end{align}
\end{itemize}

The DOFs for $W_h^{k-1}(K)$ can be given as follows.
\begin{itemize}
\item Interior DOFs $M_{K}(u)$ in  the element $K$: 
\begin{align*}
	&\quad\quad \quad   M_{K}(u)=\left\{\int_{K}  u\cdot   v \mathrm{d} V \text{ for all } v \in  P_{k-1}(K) \right\}
\end{align*}
\end{itemize}

\begin{lemma}
	The DOFs for  $\bm \Sigma_h^{k,+}(K)$ are unisolvent. 
\end{lemma}
\begin{proof}
 The case of $k=1$ is proved in \cite[Lemma 4.3]{guzman2018inf}, and the case of $k=2$ can be proved similarly. We only prove the lemma for $k\geq 3$.
	For $\bm u\in\bm \Sigma_h^{k,+}(K)$, rewrite $\bm u=\bm w+\sum_i^{N_{k-1}}b_i\mathring{\bm \beta}_i$ with $b_i\in \mathbb R$, $\bm w\in \bm P_k(K)$, and $\mathring{\bm\beta}_i\in \mathring B^k$. Suppose that the DOFs \eqref{vector-sigma-DOFs1}--\eqref{vector-sigma-DOFs4} vanish on  $\bm u$. We must show that $\bm u=0$. Since $\mathring{\bm\beta}_i$ vanish on $\partial K$, $\bm w$ vanishes on $\partial K$ by the DOFs in \eqref{vector-sigma-DOFs1}--\eqref{vector-sigma-DOFs3}. The DOFs in the second set of \eqref{vector-sigma-DOFs4} leads to $\nabla\cdot\bm u=0$ since $\nabla\cdot\bm u\in P_{k-1}(K)\slash \mathbb R$. Therefore $\bm u=\nabla\times \bm v$ with $\bm v\in \mathring V^{r-1,k+1}_h(K)$. Using the DOFs in the first set of \eqref{vector-sigma-DOFs4}, we obtain $\bm u=0.$
	\end{proof}

\begin{lemma}\label{unisolvence}
The DOFs for  $V^{r-1, k+1}_{h}(K)$ are unisolvent. 
\end{lemma}
\begin{proof}
Since the complex \eqref{local-complex} is exact, we have
\begin{align}
	\dim{V_h^{r-1,k+1}(K)}&=\dim{\bm \Sigma_h^{k,+}(K)}+\dim{\Sigma_h^r(K)}-\dim{W_h^{k-1}(K)}-1.
\end{align}
We can check that the space of the DOFs has the same dimension. Then it suffices to show that if all the DOFs vanish on $\bm u\in V^{r-1,k+1}_{h}(K)$, then $\bm{u}=0$. To see this, we first show that $\nabla\times \bm{u}=0$. Using the properties of the Poincar\'{e} operators, we have $\nabla\times\mathfrak p^2 \bm \Sigma_h^{k,+}(K)\subset \bm \Sigma_h^{k,+}(K)$. By integration by parts, the following DOFs for $\bm\Sigma_h^{k,+}(K)$ vanish on $\nabla\times \bm{u}$: 
\[\int_{f_i}\nabla\times \bm u\cdot\bm n_i\d A=\int_{\partial f_i} \bm u\cdot\bm \tau_{\partial f_i}\d s=0,\]
and 
\[\int_K\nabla\times\bm u\cdot \nabla v\d V=\int_{\partial K}\nabla\times\bm u\cdot\bm n_{\partial K} v\d A=0\text{ for any } v\in P_{k-1}(K).\]
By the unisolvence of the DOFs for $\bm\Sigma_h^{k,+}(K)$, we get 
$\nabla\times\bm u=0 \text{ in } K.$ 

Therefore on each $f_i$, there exists a $ \phi_i\in P_r(f_i)$ such that $\bm n_i\times\bm{u}|_{f_i}\times\bm n_i=\nabla_{f_i}\phi_i$. Here $\nabla_{f_i}$ is the face gradient on $f_i$.  By the edge DOFs of $V^{r-1, k+1}_{h}(K)$, we get
$\bm u\cdot \bm {\tau}_i=0$ on the edge $\bm e_i$. Therefore $\phi_i$ is a constant on all the edges of $f_i$. Without loss of generality, we can choose this constant to be zero. Then  $\phi_i$ has the form $\phi_i=B_i|_{f_i}\psi_i$ with $\psi_i\in P_{r-3}(f_i)$. By the property of Koszul operators in 2D \cite[Theorem 7.1]{arnold2018finite}, for any function $\psi_i\in {P}_{r-3}(f_{i})$, there exists $\bm{q}_{i}\in  {P}_{r-3}(\hat f_{i})B_K\hat{\bm{x}}_{\hat f_{i}}$ satisfying $\bm{q}_{i}\perp \bm n_i$ and $\nabla_{f_{i}}\cdot\bm{q}_{i}=\psi_{i}$. By the DOFs in \eqref{tridef1-3}, we have
$$
0=(\boldsymbol{u}, \boldsymbol{q}_{i})_{f_{i}}=-(\phi_i, \nabla_{f_{i}} \cdot \boldsymbol{q}_{i})_{f_{i}}=-\left(B_i|_{f_i}\psi_i, \psi_i\right)_{f_{i}}.
$$
This implies that $\psi_i=0$, i.e., 
$\boldsymbol{u} \times \boldsymbol{n}_{i}=0$ on $f_i$. 

Since $\nabla\times\bm u=0$ and $\boldsymbol{u} \times \boldsymbol{n}_{i}=0$ on $f_i$, there exists $\phi=B_0\psi$ with $\psi\in P_{r-4}(K)$ such that $\bm u=\nabla \phi$. We choose $\bm{q}\in P_{r-4}(K)B_K\hat{\bm x}$ such that $\nabla\cdot\bm q=\psi$. Then 
\[0=\left(\bm u,\bm q\right)=\left(\nabla \phi,\bm q\right)=-\left(\phi,\nabla\cdot\bm q\right)=-\left(B_0{\psi},{\psi}\right).\]
This implies that $\psi=0$ and hence $\phi=0$ and $\bm u=0$.

\end{proof}

Equipping the local spaces with the above DOFs, we obtain the global finite element spaces $\Sigma_h^{r}$, $V_h^{r-1,k+1}$, $\bm \Sigma_h^{k,+}$, and $W_h^{k-1}$. 

\begin{lemma}\label{conformity} The following conformity holds:
$$
V^{r-1, k+1}_{h}\subset H(\grad\curl; \Omega).
$$
\end{lemma}
\begin{proof}
If we can verify $V^{r-1, k+1}_{h}\subset H(\curl; \Omega),$ then the conformity follows from $\nabla\times V^{r-1, k+1}_{h}\subseteq \bm{\Sigma}^{k,+}_h\subset \bm H^{1}(\Omega)$. To this end, we must show $\bm u\times\bm n_i=0$ for all $f_i\in\mathcal F_h(K)$ if the DOFs  \eqref{tridef1-1}-\eqref{tridef1-3} vanish on $\bm u\in V^{r-1, k+1}_{h}(K)$. From the vanishing DOFs involving $\nabla\times\bm u$ and $\int_{f_i}\nabla\times\bm u\cdot\bm n_i\d A=\int_{\partial f_i} \bm u\cdot\bm \tau_{\partial f_i}\d s=0$, we have $\nabla\times\bm u=0$ on  $\partial K$. Proceeding as in the proof of Lemma \ref{unisolvence}, we can show that $\bm u\times\bm n_i=0$ on each $f_i$.

\end{proof}

\begin{remark}
When $r=k$ with $k\geq 1$ and $r=k+1,k+2$ with $k=1,2$, without modifying the definition of Poincar\'e operator $\mathfrak p^2$,  the space $V_h^{r-1,k+1}$ is non-conforming in $H(\curl;\Omega)$, but $\nabla\times V_h^{r-1,k+1}$ is conforming in $\bm H^1(\Omega)$. The elements in these cases still work. See the numerical elements in Section \ref{numericalexperiment}.
\end{remark}


\section{Global Finite element complexes}\label{sec:dofs}
We now present properties of the complex \eqref{discrete-complex-r-k} with the global finite element spaces. The first property we will show is the surjectivity of $\nabla\cdot: \bm \Sigma_h^{k,+}\rightarrow W_h^{k-1}$. To this end, we need the following property for the local complex.

\begin{lemma}\label{local-ontoness}
For any $q\in W_{h}^{k-1}(K)\cap \mathring L^2(K)$, there exists $\bm v\in \bm \Sigma_{h}^{k,+}(K)\cap \bm H_0^1(K)$ such that  $\nabla \cdot \bm v=q$ and $\|\bm v\|_{1,K}\leq C\|q\|_{K}$.
\end{lemma}
\begin{proof}
	For a fixed $q\in W_{h}^{k-1}(K)\cap \mathring L^2(K)$, there exists $\bm w\in \bm H_0^1(K)$ such that (see e.g. \cite[Corollary 2.4]{Girault2012Finite})
	\[\nabla\cdot \bm w=q \text{ in }\Omega.\]
Let $\bm v\in \bm \Sigma_{h}^{k,+}(K)$ be the unique function that satisfies
\begin{align*}
	\int_K\bm v\cdot \nabla p\d V=\int_K\bm w\cdot \nabla p\d V, \ \forall 
	p \in P_{k-1}(K),
\end{align*}
with the remaining DOFs in \eqref{vector-sigma-DOFs1}-\eqref{vector-sigma-DOFs4} vanishing on $\bm{v}$.
Then $\bm v\in \bm \Sigma_{h}^{k,+}(K)\cap \bm H_0^1(K).$ Moreover, integrating by parts, we have
\begin{align*}
	&(\nabla\cdot\bm v, p)=(\bm v,\nabla p)=(\bm w,\nabla p)=(\nabla\cdot\bm w, p)=(q, p), \ \forall 
	p \in P_{k-1}(K)\slash \mathbb R,\\
	&(\nabla\cdot\bm v, 1)=\langle \bm v\cdot\bm n,1\rangle=0=(q,1).
\end{align*}
This implies $\nabla\cdot\bm v-q=0$ since $\nabla\cdot\bm v-q\in P_{k-1}(K)$.

We now prove  $\|\bm v\|_{1,K}\leq C\|q\|_{K}$ by a scaling argument. Denote \[n_{k-1}=\dim P_{k-1}(K),\] we can express $\bm v$ as
\[\bm v=\sum_{i=2}^{n_{k-1}}(\bm w,\nabla p_i)\bm N_i,\] 
where $\{p_i\}_{i=2}^{n_{k-1}}$ is a set of basis functions of $P_{k-1}(K)\slash \mathbb R$ and $\bm N_i$ is the dual basis of $p_i$ with respect to the DOFs $(\bm w,\nabla p_i)$, i.e., $(\bm N_i, \nabla p_j)=\delta_{ij}$. Setting $\hat{\bm v}=\det(B_K)B_K^{-1}\bm v\circ F_K$ and $\hat p=p\circ F_K$ with $B_K$ and $F_K$ defined in \eqref{mapping-domain}, we obtain
\begin{align*}
	&\|\bm v\|^2_{1,K}\leq Ch_K^{-3}\|\hat{\bm v}\|^2_{1,\hat K}\leq Ch_K^{-3}\sup_{2\leq i\leq n_{k-1}}|(\hat{\bm w},\hat{\nabla} \hat p_i)|^2\\
	=Ch_K^{-3}&\sup_{2\leq i\leq n_{k-1}}|(\hat{\nabla}\cdot\hat{\bm w}, \hat p_i)|^2\leq Ch_K^{-3}\|\hat{\nabla}\cdot\hat{\bm w}\|^2_{\hat K}
	\leq C\|{\nabla}\cdot{\bm w}\|^2_{K}=C\|q\|^2_{K}.
\end{align*}
\end{proof}

\begin{lemma}\label{global-ontoness}
	For any $q\in W_{h}^{k-1}$, there exists $\bm v\in \bm \Sigma_{h}^{k,+}$ such that  $\nabla \cdot \bm v=q$ and $\|\bm v\|_{1}\leq C\|q\|$.
\end{lemma}
\begin{proof}
Given $q\in W_{h}^{k-1}\subset \bm L^2(\Omega)$, according to \cite[Theorem 2]{arnold2020complexes}, there exists $\bm w\in \bm H^1(\Omega)$ satisfying
			$\nabla\cdot \bm w=q$ and $\|\bm w\|_{1}\leq C\|q\|$. Let $\bm I_h\bm w\in \bm \Sigma_h^k\subset\bm \Sigma_h^{k,+}$ denote the Scott-Zhang interpolation of $\bm w$ (see \cite[(2.13)]{scott1990finite} for its definition), where $\bm \Sigma_h^k$ is the vector-valued Lagrange finite element space of degree $k$. 
We also let $\bm v_1\in \bm \Sigma_{h}^{k,+}$ be the unique function that satisfies
\begin{align*}
	\int_{f_i}\bm v_1\cdot\bm n_i \d A=\int_{f_i}(\bm w-\bm I_h\bm w)\cdot\bm n_i \d A,\  \forall f_i\in \mathcal F_h,
\end{align*}
with other DOFs in \eqref{vector-sigma-DOFs1}-\eqref{vector-sigma-DOFs4} vanishing on $\bm{v}_{1}$.
Then we have, for any $K\in \mathcal T_h$,
\[(\nabla \cdot \bm v_1+\nabla \cdot \bm I_h\bm w,1)_K=\langle\bm v_1\cdot \bm n+\bm I_h\bm w\cdot\bm n,1\rangle_{\partial K}=\langle\bm w\cdot \bm n,1\rangle_{\partial K}=(\nabla \cdot \bm  w,1)_K=(q,1)_K,\]
 which means $(q-\nabla \cdot \bm v_1-\nabla \cdot \bm I_h\bm w)|_K\in W_{h}^{k-1}(K)\cap \mathring L^2(K)$. By Lemma \ref{local-ontoness}, there exists $\bm v_{2,K}\in \bm \Sigma_{h}^{k,+}(K)\cap \bm H_0^1(K)$ such that 
 $$
 	\nabla\cdot \bm v_{2,K}=(q-\nabla \cdot \bm v_1-\nabla \cdot \bm I_h\bm w)|_K, \ \forall K\in \mathcal T_h
 	$$
 	 and
 	 $$
 	\|\bm v_{2,K}\|_{1,K}\leq C(\|\bm v_1\|_{1,K}+\|\bm I_h\bm w\|_{1,K}+\|q\|_{K}).
$$
Define $\bm v_2\in \bm H_0^1(\Omega)\cap \bm\Sigma_h^{k,+}$ by $\bm v_2|_{K}=\bm v_{2,K}.$ Setting $\bm v=\bm v_1+\bm v_2+\bm I_h\bm w$, we have 
\[\nabla\cdot \bm v= \nabla\cdot(\bm v_1+\bm v_2+\bm I_h\bm w)=q \text{ and }\|\bm v\|_{1}\leq C(\|\bm v_1\|_{1}+\|\bm I_h\bm w\|_{1}+\|q\|).\]
We apply the same scaling argument as used in Lemma \ref{local-ontoness} and the approximation property of the Scott-Zhang interpolation $\bm I_h\bm w$ \cite[(4.1)]{scott1990finite} to obtain 
\begin{align*}
	\|\bm v_1\|^2_{1,K}\leq &Ch_K^{-3}\|\hat{\bm v}_1\|^2_{1,\hat K}\leq Ch_K^{-3}\big|\langle({\bm w}-{\bm I_h\bm w})\cdot{\bm n}_i,1\rangle_{\partial K}\big|^2\leq Ch_K^{-1}\|{\bm w}-{\bm I_h\bm w}\|_{\partial K}^2\\
	&\leq C\left(h_K^{-2}\|{\bm w}-{\bm I_h\bm w}\|_{K}^2+\|{\bm w}-{\bm I_h\bm w}\|_{1, K}^2\right)\leq C\|\bm w\|_{1,\omega(K)}^2
	\end{align*}
	with $\omega(K)=\text{Int}\left\{\bar K_i|\bar K_i\cap \bar K\neq \emptyset, K_i\in \mathcal T_h\right\}$. 
	Summing over $K\in \mathcal T_h$, we obtain
	\[\|\bm v_1\|_1\leq C\|\bm w\|_1,\]
	which together with $\|\bm I_h\bm w\|_{1}\leq C\|\bm w\|_{1}$ \cite[(4.5)]{scott1990finite} and $\|\bm w\|_1\leq C \|q\|$ leads to 
	\[\|\bm v\|_1\leq C\|q\|.\]
\end{proof}
\begin{corollary}\label{cor:inf-sup}
	The inf-sup condition for the Stokes problem holds, i.e., there exists a positive constant $\alpha>0$ not depending on $h$, such that
	\[\sup _{0\neq \bm v\in \bm \Sigma_h^{k,+}}\frac{(\nabla\cdot \bm v,q)}{\|\bm v\|_1}\geq \alpha\|q\|, \ \forall q\in W_h^{k-1}.\]
\end{corollary}

Corollary \ref{cor:inf-sup} implies that $\bm \Sigma_h^{k,+}-W_h^{k-1}$ leads to convergent algorithms for solving the Stokes problem with a precise divergence-free condition.

\begin{theorem}\label{exactness}
The complex \eqref{discrete-complex-r-k}  is  exact on contractible domains.
\end{theorem}
\begin{proof}
The exactness at $\Sigma_h^r$ and $V^{r-1,k+1}_h$ follows from the exactness of the standard finite element differential forms (e.g., \cite{arnold2018finite}). 
The exactness at $W^{k-1}_h$, i.e., the surjectivity of $\nabla\cdot: \bm\Sigma_h^{k,+} \to W^{k-1}_h$ is verified in Lemma \ref{global-ontoness}.

 Finally, the 
 the exactness at $\bm \Sigma_h^{k,+}$ follows from a dimension count. Let $\mathcal V$, $\mathcal E$, $\mathcal F$, and $\mathcal K$ denote the number of vertices, edges, faces, and 3D cells, respectively. Then we have 
 \[\dim \Sigma^{r}_h=\mathcal V+(r-1)\mathcal E+\frac{1}{2}(r-2)(r-1)\mathcal F+\frac{1}{6}(r-3)(r-2)(r-1)\mathcal K,\]
\[\dim W_h^{k-1}=\frac{k(k+1)(k+2)}{6}\mathcal K.\]
From the DOFs \eqref{tridef1-1} -\eqref{tridef1-4}, 
\begin{align*}
	&\dim V^{r-1, k+1}_h-\dim \bm \Sigma^{k,+}_h=r\mathcal E+\frac{1}{2}(r-2)(r-1)\mathcal F\\
	&-\mathcal F+\frac{1}{6}\big[(r-3)(r-2)(r-1)-k(k+1)(k+2)+6\big]\mathcal K. 
	\end{align*}
From the above dimension count, we have 
\[-1+\dim \Sigma_h^r - \dim V^{r-1,k+1}_h +\dim \bm \Sigma^{k,+}_h-\dim W_h^{k-1}=0,\]
where we have used Euler's formula $\mathcal V-\mathcal E+\mathcal F-\mathcal K=1$.
This completes the proof.
\end{proof}
\begin{remark}
The finite element spaces with vanishing boundary conditions also form an exact complex on contractible domains: 
\begin{equation}\label{discrete-complex-r-k-0}
\begin{tikzcd}
0\arrow{r}{\subset} &  
\mathring\Sigma_h^r  \arrow{r}{\nabla} & \mathring V^{r-1,k+1}_{h}  \arrow{r}{\nabla\times} & \mathring{\bm \Sigma}^{k,+}_h\arrow{r}{\nabla\cdot} &\mathring W_h^{k-1} \arrow{r}&0,
 \end{tikzcd}
\end{equation}
where $\mathring\Sigma_h=\Sigma^{r}_h\cap H_0^1(\Omega)$, $\mathring V^{r-1,k+1}_{h}=V^{r-1,k+1}_h\cap H_0(\grad\curl;\Omega)$, $\mathring{\bm \Sigma}^{k,+}_h={\bm \Sigma}^{k,+}_h\cap \bm H_0^1(\Omega)$, $\mathring W_h^{k-1}=W^{k-1}_h\cap \mathring L^2(\Omega)$ with  $H_0(\grad \curl; \Omega)=\{\bm u\in H(\grad \curl; \Omega):\ \bm u\times\bm n=0\text{ and } \nabla\times\bm u=0 \mbox{ on } \partial \Omega\}$.
\end{remark}

For $\delta >0$, 
denote $\Sigma=H^{3/2+\delta}(\Omega)$ and $V=\{\bm u\in \bm H^{1/2+\delta}(\Omega):\ \nabla \times \bm u \in H^{3/2+\delta}(\Omega)\}$. 
We use $\pi_h: \Sigma\rightarrow \Sigma_h^{r}$, $\widetilde{\bm \pi}_h: \bm \Sigma\rightarrow \bm \Sigma_h^{k,+}$,  $\bm r_h: V\rightarrow V_h^{r-1,k+1}$, and $i_h: L^2(\Omega)\rightarrow W_h^{k-1}$ to denote the interpolation operators defined by the DOFs for $\Sigma_h^r$, $\bm \Sigma_h^{k,+}$, $V_h^{r-1,k+1}$, and $W_h^{k-1}$, respectively.

We summarize the interpolations defined above in the following diagram:
\begin{equation}\label{2complex}
\begin{tikzcd}
\mathbb{R} \arrow[r,"\subset"]  & H^1(\Omega) \arrow[d ]\arrow[r,"\nabla"]  & H(\grad \curl;\Omega)\arrow[r,"\nabla\times"]\arrow[d]& \bm H^1(\Omega)\arrow[r,"\nabla\cdot"]\arrow[d]  & L^2(\Omega)\arrow[r]\arrow[d ]& 0\\
\mathbb{R} \arrow[r,"\subset"]  & \Sigma \arrow[d, "\pi_h" ]\arrow[r,"\nabla"]  & V\arrow[r,"\nabla\times"]\arrow[d, "\bm r_h" ]  &\bm{\Sigma}\arrow[r]\arrow[d, "\widetilde{\bm\pi}_h" ]\arrow[r,"\nabla\cdot"]\arrow[d]  & L^2(\Omega) \arrow[r]\arrow[d ,"i_h"]& 0\\
\mathbb{R} \arrow[r,"\subset"]  & \Sigma^{r}_h\arrow[r,"\nabla"]  & V^{r-1,k+1}_h\arrow[r,"\nabla\times"]  & \bm\Sigma^{k,+}_h\arrow[r,"\nabla\cdot"]  & W_h^{k-1} \arrow[r]& 0.\end{tikzcd}
\end{equation}

By a similar argument as in \cite[Theorem 5.49]{Monk2003},  the interpolations in \eqref{2complex} commute with the differential operators. 
\begin{lemma}\label{commute}
The last two rows of the complex \eqref{2complex} are a commuting diagram, i.e.,
\begin{align}
	\nabla\pi_h u&=\bm r_h\nabla u \text{ for all } u\in \Sigma,\label{Pih_and_pih}\\
	\nabla\times\bm r_h \bm u&=\widetilde{\bm \pi}_h\nabla\times\bm u \text{ for all } \bm u\in V,\label{Pih_and_tildepih}\\
	\nabla\cdot\widetilde{\bm \pi}_h \bm u&=i_h\nabla\cdot\bm u \text{ for all } \bm u\in \bm \Sigma. \label{wh_and_tildepih}
\end{align}
\end{lemma}

We adopt the following Piola mapping to transform the finite element function $\bm u$ on a general element $K$ to a function $\hat{\bm u}$ on the reference element $\hat K$:
\begin{align}
\bm u \circ F_K = {B_K^{-T}}\hat{\bm u}.\label{mapping-u}
\end{align}
By a simple calculation, we have
\begin{align}
(\nabla\times\bm u) \circ F_K &= \frac{B_K}{\det(B_K)}\hat{\nabla}\times\hat{ \bm u},\label{mapping-curlu}\\
\bm n_i\circ F_K&= \frac{B_K^{-T} \hat {\bm n}_i}{\left|B_K^{-T} \hat {\bm n}_i\right|},\label{n}\\
\bm \tau_i\circ F_K&= \frac{B_K \hat {\bm \tau}_i}{\left|B_K\hat {\bm \tau}_i\right|}.\label{tau}
\end{align}

The following lemma relates the interpolation on $K$ to that on $\hat K$. 

\begin{lemma}\label{relation}
For $\bm u\in W$,  we have
$\widehat{\bm r_K \bm u}=\bm r_{\hat K}\hat {\bm u}$ with the transformation \eqref{mapping-u}.
\end{lemma}
\begin{proof}
Following \cite[Proposition 3.4.7]{brenner2008mathematical}, we only need to show the DOFs for defining $\widehat{\bm r_K \bm u}$ are linear combinations of those for defining $\bm r_{\hat K}\hat {\bm u}$. 

  By the transformations \eqref{mapping-u}, \eqref{mapping-curlu}, \eqref{n}, and \eqref{tau}, we have that all the DOFs in \eqref{tridef1-1}--\eqref{tridef1-4} are linear combinations of those for $\hat{\bm u}$ on $\hat K$. For instance, \begin{align*}
	&\int_{f_i}\nabla\times{\bm u}\cdot \bm \tau_i^1 \d A=\frac{1}{\det(B_K)}\int_{\hat f_i}\hat \nabla\times\hat{\bm u}\cdot \bm B_K^T\bm \tau_i^1 \d \hat A\\
	=&\frac{1}{\det(B_K)}\int_{\hat f_i}\hat \nabla\times\hat{\bm u}\cdot \left((\bm B_K^T\bm \tau_i^1\cdot \hat{\bm \tau}_i^1)\hat{\bm \tau}_i^1+(\bm B_K^T\bm \tau_i^1\cdot \hat{\bm \tau}_i^2)\hat{\bm \tau}_i^2+(\bm B_K^T\bm \tau_i^1\cdot \hat{\bm n}_i)\hat{\bm n}_i\right) \d \hat A\\
	=&\frac{1}{\det(B_K)}\int_{\hat f_i}\hat \nabla\times\hat{\bm u}\cdot \left((\bm B_K^T\bm \tau_i^1\cdot \hat{\bm \tau}_i^1)\hat{\bm \tau}_i^1+(\bm B_K^T\bm \tau_i^1\cdot \hat{\bm \tau}_i^2)\hat{\bm \tau}_i^2\right) \d \hat A\\
	&+\frac{(\bm B_K^T\bm \tau_i^1\cdot \hat{\bm n}_i)}{\det(B_K)}\int_{\partial\hat f_i}\hat{\bm u}\cdot \hat{\bm \tau}_{\partial f_i} \d \hat s.
\end{align*}
This completes the proof.
\end{proof}

Next, we establish the approximation property of the interpolation operators.

\begin{theorem}\label{inter-est}
Assume that $\bm u\in \bm H^{s+(r-k-1)}(\Omega)$ and $\nabla\times\bm u\in  \bm H^s(\Omega)$, $s \geq 3/2+\delta $ with $\delta>0$, and $r=k$, $k+1$, or $k+2$.
Then we have the following error estimates for the interpolation $\bm r_h$,
	\begin{align}
	&\left\|\bm u-\bm r_h\bm u\right\|\leq Ch^{\min\{s+(r-k-1),r\}}(\left\|\bm u\right\|_{s+(r-k-1)}+\left\|\nabla\times\bm u\right\|_{s}),\label{inter-u}\\
	&\left\|\nabla\times(\bm u-\bm r_h\bm u)\right\|\leq Ch^{\min\{s,k+1\}}\left\|\nabla\times\bm u\right\|_{s},\label{inter-curlu}\\
	&	\left|\nabla\times(\bm u-\bm r_h\bm u)\right|_1\leq Ch^{\min\{s-1,k\}}\left\|\nabla\times\bm u\right\|_{s}.\label{inter-curlcurlu}
	\end{align}
\end{theorem}
\begin{proof}
By the identity $\widehat{\bm r_K \bm u}=\bm r_{\hat K}\hat {\bm u}$  and the inclusion $\bm P_{r-1}(K)\subseteq V_h^{r-1,k+1}(K)$  (Lemma~\ref{relation} and Lemma~\ref{Vh}), the proof is standard, c.f., \cite[Theorem 5.41]{Monk2003}. Here we have used Lemma \ref{commute} to prove \eqref{inter-curlu} and \eqref{inter-curlcurlu}.
\end{proof}

\section{Applications to $-\curl\Delta\curl$ problems}\label{numericalexperiment}
In this section, we use the three  $\grad\curl$-conforming finite element families to solve a problem with $\curl\Delta\curl$ operator: for $\bm  f\in H(\div^0;\Omega)$, find $\bm u$, such that
\begin{equation}\label{prob1}
\begin{split}
-\nabla\times\Delta (\nabla\times\bm u)+\bm u&=\bm f\ \ \text{in}\;\Omega,\\
\nabla \cdot \bm u &= 0\ \ \text{in}\;\Omega,\\
\bm u\times\bm n&=0\ \ \text{on}\;\partial \Omega,\\
\nabla \times \bm u&=0\ \  \text{on}\;\partial \Omega.
\end{split}
\end{equation}
Here $\bm n$ is the unit outward normal vector on $\partial \Omega$, and
 $H(\div^0;\Omega)$ is the space of $\bm L^2(\Omega)$ functions with vanishing divergence, i.e., 
\[H(\text{div}^0;\Omega) :=\{\bm u\in {\bm L}^2(\Omega):\; \nabla\cdot \bm u=0\}.\]
Taking divergence on both sides of the first equation of \eqref{prob1}, we see that $\nabla\cdot\bm u=0$ automatically holds with $\bm{f}\in H(\div^0;\Omega)$.


The variational formulation reads:
find $\bm u\in H_0(\grad\curl;\Omega)$,  such that
\begin{equation}\label{prob22}
\begin{split}
a(\bm u,\bm v)&=(\bm f, \bm v)\quad \forall \bm v\in H_0(\grad\curl;\Omega),
\end{split}
\end{equation}
with $a(\bm u,\bm v):=(\nabla\nabla\times\bm u, \nabla\nabla\times\bm v) + (\bm u,\bm v)$. The weak form \eqref{prob22} can be regarded as a model problem for the high order problems in MHD, e.g., \cite[(1)]{chacon2007steady} and continuum mechanics with size effects, e.g., \cite[(3.27)]{mindlin1962effects}, \cite[(35)]{park2008variational}. 
\begin{remark}
The $\grad\curl$ operator appears in the following complex (referred to as the $\grad\curl$ complex) 
\begin{equation}\label{grad-curl}
\begin{tikzcd}[column sep=1.5em]
0 \!\arrow{r}{}\!&\!H^{q}\arrow{r}{\grad} & H^{q-1}\otimes \mathbb{V}  \arrow{r}{\grad\curl} &[12] H^{q-3}\otimes \mathbb{T}\! \arrow{r}{\curl}\! & H^{q-4}\otimes \mathbb{M} \!\arrow{r}{\div}\! &\! H^{q-5}\otimes \mathbb{V} \arrow{r}\! & \!0,
\end{tikzcd}
\end{equation}
which can be derived from de Rham complexes \cite[(46)]{arnold2020complexes}. Thus \eqref{prob22} is closely related to one of the Hodge-Laplacian problems associated to the $\grad\curl$ complex.
\end{remark}

 \begin{remark}
 With the given boundary conditions and the identity for vector Laplacian $-\Delta \bm{u}=-\nabla\nabla\cdot \bm{u}+\nabla\times \nabla\times \bm{u}$, the above weak form is equivalent to the quad-curl problem, i.e., 
 $
(\nabla\nabla\times\bm u, \nabla\nabla\times\bm v) =(\nabla\times \nabla\times\bm u, \nabla\times \nabla\times\bm v).
 $
 \end{remark}
The $\grad\curl$-conforming finite element method for \eqref{prob22} reads: seek $\bm u_h\in \mathring V^{r-1,k+1}_{h} $,  such that
\begin{equation}\label{prob3}
\begin{split}
 a(\bm u_h,\bm v_h)&=(\bm f, \bm v_h)\quad \forall \bm v_h\in \mathring V^{r-1,k+1}_{h} .
\end{split}
\end{equation}


\begin{theorem}\label{regularity-3D}
We assume that $\Omega$ is a simply-connected Lipschitz polyhedral domain with a connected boundary. There exists a constant $\alpha>1/2$ such that the solution $\bm u$ of \eqref{prob1} satisfies
	\[\bm u\in  \bm H^{\alpha}(\Omega),~ \nabla\times\bm u\in \bm H^{1+\alpha}(\Omega),\] and it holds 
	\[\|\bm u\|_{{\alpha}}+\|\nabla\times\bm u\|_{{1+\alpha}}\leq C\|\bm f\|.\]
\end{theorem}
\begin{proof}
The claim that $\bm u\in \bm H^{\alpha}(\Omega)$ follows from the embedding
$H_0(\curl;\Omega)\cap H(\div;\Omega)\hookrightarrow \bm H^{\alpha}(\Omega)$ with  $\alpha>1/2$ \cite{amrouche1998vector}, and it holds
\[\|\bm u\|_{{\alpha}}\leq C\left( \|\bm u\|+\|\nabla\cdot \bm u\|+\|\nabla\times\bm u\|\right)=C\left( \|\bm u\|+\|\nabla\times\bm u\|\right).\]
Furthermore, by Poincar\'e inequality, we have
\[\|\bm u\|_{{\alpha}}\leq C\left( \|\bm u\|+\|\nabla\nabla\times\bm u\|\right)\leq C\|\bm f\|.\]
If $-\Delta(\nabla\times\bm u)$ belongs to $\bm L^2(\Omega)$, then from the boundary condition
$\nabla\times\bm u=0$ and
the regularity of the Laplace problem \cite[Theorem 3.18]{Monk2003}, we can obtain 
$\nabla\times\bm u\in \bm H^{1+\alpha}(\Omega)$ with $\alpha>1/2$, and 
\[\|\nabla\times\bm u\|_{{1+\alpha}}\leq C \|\Delta(\nabla\times\bm u)\|.\]

It suffices to show that $(\nabla\times)^3 \bm u\in \bm L^2(\Omega)$ and $\|(\nabla\times)^3 \bm u\|\leq C\|\bm f\|$ since $-\Delta (\nabla\times\bm{u})=-\nabla\nabla\cdot \nabla\times\bm{u}+(\nabla\times)^3\bm{u}=(\nabla\times)^3\bm{u}$. If we can prove 
	\begin{align}\label{boundedfunctional}
	g(\bm v):=((\nabla\times)^2\bm u,\nabla\times\bm v)\leq C_0\|\bm v\|,\text{ for all } \bm v\in H_0(\curl;\Omega),
		\end{align}
		where $H_0(\curl;\Omega)=\{\bm u\in H(\curl;\Omega):\bm u\times\bm n=0 \text{ on }\partial\Omega\}$, then, by Hahn Banach theorem, there is a unique extension of the map $g(\bm v)$ to a bounded linear functional from all of $\bm L^2(\Omega)$ to $\mathbb R$ with the bound $C_{0}$. Moreover, by Riesz representation theorem, there exists a unique element $\bm \phi\in \bm L^2(\Omega)$ such that
    \[g(\bm v)=((\nabla\times)^2\bm u,\nabla\times \bm v)=(\bm \phi,\bm v), \text{ for }\bm v\in H_0(\curl;\Omega).\]
   From the definition of the adjoint of $\nabla\times$, we have $(\nabla\times)^3\bm u=\bm \phi\in \bm L^2(\Omega)$ and $\|(\nabla\times)^3\bm u\|=\|g\|_{\mathcal L(\bm L^2(\Omega),\mathbb R)}\leq C_{0}.$
    
	To prove \eqref{boundedfunctional}, we first seek $q\in H_0^1(\Omega)$ such that
	\begin{align*}
		-\Delta q=\nabla\cdot \bm v\in H^{-1}(\Omega).
	\end{align*}
	Then it holds $\|\nabla q\|\leq \|\bm v\|.$
	Applying \cite[Theorem 3.6]{Girault2012Finite} to $\bm v-\nabla q$, there exists a divergence-free vector potential $\bm w\in H_0(\curl;\Omega)$ satisfying
	\begin{align}\label{vcurlw}
	\bm v-\nabla q=\nabla\times\bm w.
	\end{align}
	Since $\bm v-\nabla q\in H_0(\curl;\Omega)$, then $\bm w\in H_0(\curl\curl;\Omega)$.
	From \eqref{vcurlw} and the Friedrichs inequality \cite[Corollary 3.51]{Monk2003}, we have
	\begin{align*}
	&((\nabla\times)^2\bm u,\nabla\times \bm v)=((\nabla\times)^2\bm u,(\nabla\times)^2\bm w)\\
	=&(\bm f-\bm u,\bm w)\leq \|\bm f-\bm u\|\|\bm w\|\leq C\|\bm f-\bm u\|\|\nabla\times\bm w\|\\
	\leq C&\|\bm f-\bm u\|\left(\|\bm v\|+\|\nabla q\|\right)\leq C\|\bm f-\bm u\|\|\bm v\|\leq C\|\bm f\|\|\bm v\|,
	\end{align*}
	which leads to \eqref{boundedfunctional} with $C_0=C\|\bm f\|.$
	\end{proof}

To estimate the error in the sense of $H(\curl)$-norm, we introduce the following auxiliary problem. Find $\bm w$ such that
\begin{equation}\label{auxiliary-prob}
\begin{split}
	-\nabla\times\Delta(\nabla\times\bm w)+\bm w&=(\nabla\times)^2(\bm u-\bm u_h)\text{ in } \Omega,\\
	\nabla\cdot\bm w&=0\text{ in } \Omega, \\
	\bm w\times\bm n&=0 \text{ on }\partial \Omega,\\
	\nabla\times\bm w&=0 \text{ on }\partial \Omega.
\end{split}
\end{equation}
Due to the special form of the right-hand side in the auxiliary problem, we can have a better regularity estimate  by a suitable modification to the proof of Theorem \ref{regularity-3D}. This result will play an important role in the dual argument in the approximation analysis below.
\begin{theorem}\label{regularity-aux}
We assume that $\Omega$ is a simply-connected Lipschitz polyhedral domain with a connected boundary. The solution $\bm w$ of \eqref{auxiliary-prob} satisfies
		\[\|\bm w\|_{{\alpha}}+\|\nabla\times\bm w\|_{{1+\alpha}}\leq C\|\nabla
		\times(\bm u-\bm u_h)\|.\]
\end{theorem}
\begin{remark}
	Furthermore, if $\Omega$ is convex, then the constant $\alpha$ in Theorem \ref{regularity-3D} and Theorem \ref{regularity-aux} can be 1. 
\end{remark}

\begin{theorem}\label{uh_approx}
For $r=k$, $r=k+1$, or $r=k+2$, if $\bm u\in \bm H^{s+(r-k-1)}(\Omega)$ and $\nabla\times\bm u\in  H^s(\Omega)$, $s \geq 1+\alpha$, we have the following error estimates for the numerical solution $\bm u_h$:
	\begin{align}
	&\quad\quad\left\|\bm u-\bm u_h\right\|_{H(\grad\curl;\Omega)}\leq Ch^{\min\{s-1,k\}}\left(\left\|\bm u\right\|_{s-1}+\left\|\nabla\times\bm u\right\|_{s}\right),\label{gradcurluh}\\
	&\quad\quad\ \left\|\nabla\times(\bm u-\bm u_h)\right\|\leq Ch^{\min\{s,k+1,2\alpha\}}\left(\left\|\bm u\right\|_{s-1}+\left\|\nabla\times\bm u\right\|_{s}\right),\label{curluh}\\
	&\left\|\bm u-\bm u_h\right\|\leq Ch^{\min\{s,k+1,2\alpha\}}(\left\|\bm u\right\|_{s}+\left\|\nabla\times\bm u\right\|_{s})\text{ when } r=k+1,k+2.\label{uh}
		\end{align}
\end{theorem}
\begin{proof}
	The estimates \eqref{gradcurluh} and \eqref{curluh} follow immediately from C\'ea's lemma, the dual argument, and Theorem \ref{inter-est}. Proceeding as in the proof of \cite[Theorem 6]{Sun2016A}, we can show that \eqref{uh} holds.	
\end{proof}

\begin{remark}
	The estimate for $\left\|\bm u-\bm u_h\right\|$ is not optimal for the family $r=k+2$. 
\end{remark}

The validity of the $\grad\curl$-conforming elements can be guaranteed by the theoretical analysis. 	
We now carry out several numerical tests to validate the nonconforming elements without the modification of the Poincar\'e operator. 
We consider the problem \eqref{prob1} on a unit cube $\Omega=(0,1)\times(0,1)\times(0,1)$ 
with an exact solution
\[\bm u= \left(\begin{array}{c}
	\sin(\pi x_1)^3\sin(\pi x_2)^2\sin(\pi x_3)^2\cos(\pi x_2)\cos(\pi x_3)\\ 
    \sin(\pi x_2)^3\sin(\pi x_3)^2\sin(\pi x_1)^2\cos(\pi x_3)\cos(\pi x_1)\\
    -2\sin(\pi x_3)^3\sin(\pi x_1)^2\sin(\pi x_2)^2\cos(\pi x_1)\cos(\pi x_2)
\end{array}\right).\]
Then, by a simple calculation, we can obtain the source term $\bm f$. We denote the finite element solution as $\bm u_h$. To measure the error between the exact solution and the finite element solution, we denote  
	\[\bm e_h=\bm u-\bm u_h.\]

For the mesh, we partition the unit cube into $N^3$ small cubes and then partition each small cube into 6 congruent tetrahedra. 

We first use the lowest-order $(k=1)$ elements in the families $r=k$ and $r=k+1$ to solve the problem \eqref{prob1} on the uniform tetrahedral mesh. Tables \ref{tetra-tab1} and \ref{tetra-tab2} illustrate errors and convergence rates for the two families.  
We observe that the numerical solution converges to the exact one at rate $h$ for the case $r=k=1$, and at rate $h^2$ for $r=k+1=2$ in the sense of the $L^2$-norm. In addition, the two families have the same convergence rate $h^2$ in the $H(\curl)$-norm and $h$ in the $H(\grad\curl)$-norm, respectively. 

We now test the third-order element ($k=3$). Tables \ref{tetra-tab4} demonstrates numerical data for the family $r=k$. 
Our code for the basis functions of the elements when $k=1,3$ is available at

\url{https://github.com/QianZhangMath/3D-gradcurl-conforming-FE-18DOF}.

\begin{table}[!ht]
	\centering
	\caption{Numerical results by the $\grad \curl$-conforming element with $r=k$ and $k=1$} \label{tetra-tab1}
	\begin{tabular}{cccccccc}
		\hline
		$N$ &$\left\|\bm e_h\right\|$&rates&$\left\|\nabla\times\bm e_h\right\|$&rates&$\left\|\nabla\nabla\times\bm e_h\right\|$& rates\\
		\hline
		$45$&8.642113e-03& &7.620755e-02&&2.862735e+00&\\
        $50$&7.401715e-03& 1.4705&6.317760e-02&1.7797&2.601358e+00&0.9087 \\
        $55$&6.443660e-03& 1.4544&5.314638e-02&1.8141&2.382186e+00&0.9235
 \\
 $60$&5.687783e-03& 1.4340&4.527838e-02&1.8414&2.196043e+00&0.9351\\
		\hline
	\end{tabular}
\end{table}

\begin{table}[!ht]
	\centering
	\caption{Numerical results by the $\grad \curl$-conforming element with $r=k+1$ and $k=1$} \label{tetra-tab2}
	\begin{tabular}{cccccccc}
		\hline
		$N$ &$\left\|\bm e_h\right\|$&rates&$\left\|\nabla\times\bm e_h\right\|$&rates&$\left\|\nabla\nabla\times\bm e_h\right\|$& rates\\
		\hline
		$ 30$ &1.334051e-02 && 1.453615e-01& &4.055510e+00& \\
		$ 35$ &1.033747e-02&1.6544& 1.135563e-01&1.6018 &3.567777e+00 & 0.8312  \\
		$ 40$&8.212073e-03&1.7237&9.077071e-02&1.6772&3.178759e+00 &0.8646\\
		$ 45$&6.662599e-03&1.7753&7.399883e-02&1.7344&2.862553e+00&0.8896\\
	  \hline
	\end{tabular}
\end{table}


\begin{table}[!ht]
	\centering
	\caption{Numerical results by the $\grad \curl$-conforming element with $r=k$ and $k=3$} \label{tetra-tab4}
	\begin{tabular}{cccccccc}
		\hline
		$N$ &$\left\|\bm e_h\right\|$&rates&$\left\|\nabla\times\bm e_h\right\|$&rates&$\left\|\nabla\nabla\times\bm e_h\right\|$& rates\\
		\hline
		$ 10$&3.047288e-04&  &2.974941e-03&& 2.909078e-01&\\
		$ 12$ &1.719285e-04& 3.1392 & 1.403569e-03&4.1202 & 1.779005e-01 &2.6973 \\
		$ 14$&1.070064e-04& 3.0761    &7.353798e-04 &4.1932& 1.162168e-01&2.7620\\
		$ 16$&7.125639e-05&3.0450&4.174453e-04&4.2405&7.986321e-02 &2.8094\\
		\hline
	\end{tabular}
\end{table}

\section{Concluding remarks}
In this paper we constructed 3D finite element Stokes complexes on tetrahedral meshes. Generalizing the modified Bernardi-Raugel bubbles in \cite{guzman2018inf} to an arbitrary order and utilizing the Poincar\'e operators for the de~Rham complexes, we obtain simple finite element spaces with canonical DOFs. The newly obtained finite elements allow further applications in mass-conservative approximation of fluid mechanics and high order models in continuum mechanics and electromagnetism. 

Since the entire discrete de~Rham complexes are obtained, one may further investigate robust solvers in the framework of subspace correction \cite{lee2007robust, schoeberlthesis}. These results also show promising directions for elasticity as in \cite{burman2020application,christiansen2020discrete,gopalakrishnan2012second}.

\bibliographystyle{plain}
\bibliography{reference}
\vspace{1cm}

\end{document}